\documentclass[11pt]{amsart}
\usepackage{amsmath}
\usepackage{graphicx}
\usepackage{epstopdf}
\usepackage{multicol}
\usepackage{amssymb,amscd,array}
\usepackage{color}
\usepackage{float}
\makeindex \makeatletter
\def\captionof#1#2{{\def\@captype{#1}#2}}
\makeatother

\newcounter{tablegroup}
\newcounter{subtable}[tablegroup]

\newtheorem{thm}{Theorem}[section]
\newtheorem{cor}[thm]{Corollary}
\newtheorem{lem}[thm]{Lemma}

\newtheorem{prop}[thm]{Proposition}
\newtheorem{Ass}[thm]{Assumption}
\newtheorem{defn}[thm]{Definition}
\newtheorem{rem}[thm]{\bf Remark}

\numberwithin{equation}{section}
\begin{document}
\title{Li-Yorke chaos for dendrite maps with zero topological
entropy and $\omega$-limit sets}
\author{Ghassen Askri   }
\address{Ghassen Askri, University of Carthage, Faculty of Sciences of
Bizerte, Department of Mathematics, Jarzouna, 7021, Tunisia}
 \email{askri.ghassen@hotmail.fr}
 %\footnote{This work was
%supported by the research unit 99UR/15-15}
%
%
%
%\subjclass[2010]{37E99, 37B99}
\keywords{Dendrite, dendrite map, $\omega$-limit set, decomposition,
Li-Yorke pair, Li-Yorke chaos}

\begin{abstract} Let $X$ be a dendrite  with  set of endpoints $E(X)$
  closed  and let $f:~X \to X$ be a continuous map with zero topological entropy.
   Let $P(f)$ be the set of periodic points of $f$. We prove that if $L$ is an infinite
    $\omega$-limit set of $f$ then $L\cap P(f)\subset E(X)^{\prime}$,
    where $E(X)^{\prime}$ is the set of all accumulations points of $E(X)$. Furthermore, if $E(X)$
is countable and $L$ is uncountable then $L\cap P(f)=\emptyset$. We
also show  that if $E(X)^{\prime}$ is finite then any uncountable
$\omega$-limit set of $f$ has a decomposition and as a consequence
if $f$ has a Li-Yorke  pair $(x,y)$ with $\omega_f(x)$ or
$\omega_f(y)$ is uncountable then $f$ is Li-Yorke chaotic.

\end{abstract}
\maketitle
\section{\bf Introduction}
\bigskip

\medskip

%Let $X$ be a topological space. We say that $J$ is a \emph{free interval} of $X$ if it is an open
%subset of $X$ homeomorphic to an open interval of the real line.

Let $X$ be a compact metric space with metric $d$ and  $f:
X\longrightarrow X$ be a continuous map. Let $\mathbb{Z}_{+}$ and
$\mathbb{N}$ be the sets of non-negative integers and positive
integers respectively. Denote by $f^{n}$ the $n$-th iterate of $f$;
that is,  $f^{0}$ is the identity map, and $f^{n} = f\circ f^{n-1}$
if $n\geq 1$. For any $x\in X$ the subset $O_{f}(x)= \{f^{n}(x): \
n\in\mathbb{Z}_{+}\}$ is called the \textit{$f$-orbit of $x$}. A
point $x\in X$  is called \textit{periodic} of prime period
$n\in\mathbb{N}$ if $f^{n}(x)=x$ and $f^{i}(x)\neq x$ for $1\leq
i\leq n-1$.
  We denote by $Fix(f)$ (resp. $P(f)$) the set of fixed
  points (resp. periodic points) of $f$. Let $A$ be a non empty subset of $X$. It is called periodic with period $p\geq 1$ if $A, f(A), \dots, f^{p-1}(A)$ are pairwise disjoint and $f^{p}(A)=A$. For any $x\in
  X$, we denote by $\omega_f (x) =\cap_{n\geq 0} \overline{O_{f}(f^n
  (x))}$ the omega limit set of $x$.
 A pair $ (a,b)$ in $X^2$ is called \textit{proximal} if  $\underset{n \to
 +\infty}{\liminf} \ d(f^{n}(a),f^{n}(b))=0 $, it is called
   \textit{distal} if   $\underset{n \to
  +\infty}{\liminf} \ d(f^{n}(a),f^{n}(b))>0$  and it is called
  \textit{asymptotic} if $\underset{n \to
  +\infty}{\limsup} \ d(f^{n}(a),f^{n}(b))=0$.
 A pair $ (a,b)$ in $X^2$ is called a \textit{Li-Yorke} pair (of $ f $) if it is proximal
 but not asymptotic. A subset $ S $ of $ X $ with at least two points is a
 \textit{scrambled }set (of $f$) if
 any proper pair $(a,b)\in
  S^{2}$ is a Li-Yorke pair. A continuous map $f:~X\to X$ is called \emph{Li-Yorke chaotic}  if
  it has an \emph{uncountable} scrambled set. Denote by $h(f)$ the topological entropy of $f$ (See \cite{AKM}, \cite{Bow}, \cite{Din}).
For any non empty  subset $F$ of a compact metric space $X$, the set
of accumulation points of $F$,  denoted  by $F^{\prime}$, is called
the \emph{derived set} of $F$. More generally, for any $n\geq 1$, we
define $F^{(n)}=(F^{(n-1)})^{\prime}$ the $n$-th derivative of $F$,
where $F^{0}=F$.  If $F$ is closed then $F^{\prime}$ is a closed
subset of $F$. A \emph{continuum} is a compact connected metric
space. An arc is any space homeomorphic to the compact interval
$[0,1]$. A topological space is \emph{arcwise} connected if any two
of its points can be joined by an arc. We use the terminology from
Nadler \cite{Nad}.
 By
a \textit{dendrite} $X$, we mean a locally connected continuum which
contains no homeomorphic copy of a circle i.e simple closed curve.
Every sub-continuum of a dendrite is a dendrite (\cite{Nad}, Theorem
10.10) and every connected subset of $X$ is arcwise connected
(\cite{Nad}, Proposition 10.9). Let $x\in  X$. The number of
connected components of $X \backslash \{x\}$, denoted  $ord(x,X)$,
is called the order of $x$ in $X$. If $ord(x,X)=1$ (resp.
$ord(x,X)=2$, resp. $ord(x,X) \geq 3 $) then $x$ is called and
\emph{endpoint} (resp. \emph{cut point}, resp. \emph{branch point})
of $X$. If there is no confusion, we denote  $ord(x)$ instead of
$ord(x,X)$. We denote by $E(X)$ (resp. $B(X)$) the set of endpoints
(resp. branch points) of $X$. A tree, is a dendrite with finitely
many endpoints.
 In addition, any two
distinct points $x,y$ of a dendrite $X$ can be joined by a unique
arc with endpoints $x$ and $y$, denote this arc by $[x,y]$ and let
$[x,y)=[x,y]\setminus\{y\}$ (resp. $(x,y]=[x,y]\setminus\{x\}$ and
$(x,y)=[x,y]\setminus\{x,y\}$). A \emph{free arc} in a dendrite is
an arc containing no branch point.

  A continuous map from a dendrite into itself is
called a \textit{dendrite map}.
  %A dendrite map $f:~X \to X$ has the \emph{periodic point property} if
%  for any subdendrite $Y \subset X$, if $Y \subset f(Y)$ then we have $Y \cap P(f) \neq
%  \emptyset$.
  For any closed subset $F$ of $X$, we call  the \emph{convex hull} of $F$,
   noted $[F]$, the intersection of all subdendrites of $X$ containing $F$.

  %denote by $[F]=\overline{\cap_{Y}
%  Y}$ where the intersection is taken of all subdendrite $Y\subset
%  X$ containing $F$, $[F]$ is called \emph{the convex hull} of $F$.

The $\omega$-limit sets play an important role in studying dynamical
systems. Sarkovski  \cite{Sar}, proved that if $f:~[0,1] \to [0,1]$
is a continuous map with zero topological entropy then any infinite
$\omega$-limit set contain no periodic point. This result remain
true for graph maps (in particular for tree maps) (\cite{HM},
Theorem 13). In this paper, we firstly study this question for
dendrite maps  with closed set of endpoints (See Theorem
\textbf{A}).  Secondly,  Smital showed in (\cite{Sm}, Theorem 3.5)
that if $L=\omega_f(x)$ is an infinite $\omega$-limit set  of an
interval map with zero topological entropy
 (in fact $L$ is uncountable) then  there is a sequence $(J_k)_{k\geq 1}$ of $f$-periodic intervals
 with the following properties: For any $k$,
 \begin{enumerate}
 \item $J_k$ has period $2^k$,
 \item $J_{k+1}\cup f^{2^k}(J_{k+1})\subset J_k$,
 \item $L\subset \cup_{i=0}^{2^k -1}f^{i}(J_k)$,
 \item $L\cap f^{i}(J_k) \neq \emptyset$ for every $i=0,1,\dots, 2^k -1$.
 \end{enumerate}
 We will extend this result to dendrite maps $f:~X \to X$ with zero topological entropy, where $E(X)$ is closed and $E(X)^{\prime}$ finite. This holds in particular if $X$ is a tree. (See Theorem \textbf{B}).\\
 \smallskip

 In the third part of the paper, we study the question: Does Li-Yorke
 pair implies chaos for dendrite maps?  Kuchta and Smital \cite{KS} proved that the existence of Li-Yorke pair   implies chaos for interval maps. In \cite{RS}, Ruette and Snoha proved that the same conclusion holds for graph maps. An example of a triangular map in the square (resp. on the Cantor set and the Warsaw circle)  answering negatively the question is found in \cite{FPS} (resp. \cite{GL}).
%it was shown that  also admit continuous selfmaps with this property.
Here we give some examples of dendrites maps with countable set of
endpoints having a Li-Yorke pair but  not Li-Yorke chaotic.

Our main results are the following:

\maketitle

%Our main result of this paper is to give a general version of Smital's Theorem (3.5 of \cite{Sm}):\\
\medskip

\textbf{Theorem A.}  Let $X$ be a dendrite with $E(X)$  closed and
let $f:~X \to X$ be a dendrite map with zero topological entropy.
Let $x\in X$. Then  we have
  \begin{enumerate}
 \item  If $\omega_f (x)$
  is infinite then $\omega_f (x) \cap P(f) \subset  E(X)^{\prime}$,
\item   If $E(X)$ is countable and $\omega_f (x)$ is
uncountable then $ \omega_f (x) \cap P(f)=\emptyset$.
\end{enumerate}

\medskip

\textbf{Remark.}
\begin{enumerate}
\item The condition  $\omega_f(x)\cap P(f)\neq \emptyset$ can occur,
we built a dendrite map $f:~X \to X$ with zero topological entropy
having an infinite $\omega$-limit set $\omega_f(x)$ containing a
periodic point,  where $E(X)$ is a closed countable set and
$\omega_f(x)$ is infinite countable. (See Example 1).
\item There is a dendrite map $f:~X\to X$ with zero topological entropy
having an  $\omega$-limit set $\omega_f(x)$ containing a periodic
point with $E(X)$  non closed and countable  and $\omega_f(x)$
uncountable. (See Example 2).

\item There is a dendrite map $f:~X \to X$ with zero topological entropy
 having an uncountable
$\omega$-limit set $\omega_f(x)$ containing a periodic point with
$E(X)$   closed and uncountable. (See Example 3).

\end{enumerate}
\medskip

\textbf{Theorem B.}\label{dec} Let $X$ be a dendrite such that
$E(X)$ is  closed set having finitely many accumulation points and
let $f:~X\to X$ be a dendrite map with zero topological entropy. Let
$L$ be an uncountable $\omega$-limit set. Then there is a sequence
of $f$-periodic subdendrites $(D_k)_{k\geq 1}$ of $X$ and a sequence
of integers $n_k\geq 2$ for every $k\geq 1$  with the followings
properties: $\forall k\geq1$,
\begin{enumerate}
\item $D_k$ has period $\alpha_k:=n_1 n_2 \dots n_k$,
\item $\cup_{k=0}^{n_j -1}f^{k \alpha_{j-1}}(D_{j}) \subset D_{j-1}$; $j\geq 2$,
\item $L \subset \cup_{i=0}^{\alpha_k -1}f^{i}(D_k)$,
\item $f(L \cap f^{i}(D_k))=L\cap f^{i+1}(D_k); 0\leq i \leq \alpha_{k}-1$. In particular $L \cap f^{i}(D_k)\neq \emptyset$, $\forall 0\leq i \leq \alpha_{k}-1$,
\item $\forall 0\leq i\neq j<\alpha_k $, $f^{i}(D_k)\cap f^{j}(D_k)$ has empty interior.
\end{enumerate}
\medskip

\textbf{Corollary 1.} Let  $X$ be a tree and $f:~X\to X$ a
continuous map with zero topological entropy. Let $L=\omega_f(x)$ an
infinite $\omega$-limit set. Then  there is a sequence $(J_k)_{k\geq
1}$ of $f$-periodic arcs
 and
a sequence of integers $n_k\geq 2$ for every $k\geq 1$ with the
following properties: For any $k$,
 \begin{enumerate}
 \item $J_k$ has period $\alpha_k:=n_1 n_2 \dots n_k$,
 \item $\cup_{k=0}^{n_j -1}f^{k \alpha_{j-1}}(J_{j}) \subset J_{j-1}$; $j\geq 2$,
\item $L \subset \cup_{i=0}^{\alpha_k -1}f^{i}(J_k)$,
\item $f(L \cap f^{i}(J_k))=L\cap f^{i+1}(J_k); 0\leq i \leq \alpha_{k}-1$. In particular $L \cap f^{i}(J_k)\neq \emptyset$, $\forall 0\leq i \leq \alpha_{k}-1$,
\item $\forall 0\leq i\neq j<\alpha_k $, $f^{i}(J_k)\cap f^{j}(J_k)$ are either disjoint or they intersect in their common endpoints.
 \end{enumerate}
 \medskip

\medskip

\textbf{Theorem C.} \label{lyc} Let $X$ be a dendrite with $E(X)$
closed and
 $E(X)^{\prime}$ finite. Let $f:~X\to X$ be a dendrite map. If $f$ has a Li-Yorke pair $(x,y)$
such that  $\omega_f (x)$ or $\omega_f(y)$ is uncountable then $f$ is Li-Yorke chaotic.\\
\medskip

\textbf{Remark.}  Theorem \textbf{C} is not always true if $X$ is a
dendrite with $E(X)$ closed countable with $E(X)^{\prime}$ infinite.
(See Example $4$).

\bigskip

   %--------------------------------------------------------------------
   %--------------------------------------------------------------------

\section{\textbf{Preliminaries}}
\bigskip

\begin{lem} \label{dist} $($\cite{MS}, Lemma $2.1)$
Let $(X,d)$ be a dendrite. Then, for every $\varepsilon>0$, there
exists $\delta=\delta(\varepsilon)>0$ such that, for any $x, y\in X$
with
 $d(x,y)\leq \delta$, the diameter \emph{diam}$([x,y])<\varepsilon$.
\end{lem}
 Here diam$(A):=\sup_{x,y\in A}d(x,y)$ where $A$ is a non empty subset
  of $(X,d)$.
%\begin{lem} \cite{Issam}  Let $[a, b]$ be a non-degenerate arc in a dendrite $(X,d)$. Then there is $\delta=\delta(a,b) > 0$ such that $[u, v]\cap[a, b] \neq \emptyset $ for any $u, v \in X$
%satisfying $d(a, u)<\delta$ and $d(b, v)<\delta$.
%\end{lem}

\begin{cor}  \label{J}
Let $(X, d)$ be a dendrite. Then, for every $x, y\in X; x\neq y$,
there is $\varepsilon > 0$ and an arc $J\subset [x,y]$ such that for
any $u\in B(x,\varepsilon) $
  and $v\in
 B(y,\varepsilon)$ we have $J \subset [u,v]$.
\end{cor}
\begin{proof}
Fix $x,y \in X; x\neq y$. Let $U$ and $V$ two disjoint subdendrites
of $X$  such that $x\in int(U)$ and $y \in int(V)$ (where $int(A)$
denote the interior of the subset $A$). Since $U, V$ and $[x,y]$ are
connected then by Theorem 10.10 of \cite{Nad}, $U\cap[x,y]$ and
$V\cap [x,y]$ are also connected. Let $c_1, c_2\in X$  satisfying
$U\cap[x,y]=[x,c_1]$ and
 $V\cap [x,y]=[c_2,y]$. We have
 $[x,c_1] \cap [y,c_2]\subset U\cap V =\emptyset$ so $[x,c_1] \cap
 [y,c_2] =\emptyset$.

Let $u\in U$ and $v\in V$ arbitrarily and denote by $u^{\prime}=
r_{[x,y]}(u), v^{\prime}=r_{[x,y]}(v)$ where $r_{[x,y]}$ is the
first point map for $[x,y]$ (See \cite{Nad}, page 176).
  We have $u^{\prime}\in [x,u]\cap [x,y]\subset U\cap [x,y]=[x,c_1]$.
  Similarly, $v^{\prime}\in [y,c_2]$. Since  $[u,u^{\prime})\cap [x,y]=
  [v,v^{\prime})\cap [x,y]=\emptyset $ (If $a=b$ then $[a,b)=\emptyset$)
  then $[u,u^{\prime}), [v,v^{\prime})$ and $[u^{\prime},v^{\prime}]$
   are pairwise disjoint so $[u,v]=[u,u^{\prime}] \cup [u^{\prime},
   v^{\prime}]\cup [v^{\prime},v]$. Finely, since $u^{\prime}\in [x,c_1],
    v^{\prime}
  \in [c_2,y]$ then $[u^{\prime},v^{\prime}]$ contains $J:=[c_1,c_2]$
  (which is independent from $u,v$) so $J\subset [u,v]$. By taking $\varepsilon>0$ with $B(x,\varepsilon)\subset U$  and $B(y,\varepsilon)\subset V$, we finish the proof of the
  corollary.
  \end{proof}

\begin{lem}  \label{most}
 Let $X$ be a dendrite and $C_1, C_2 $ two disjoint connected subsets
 in $X$. Then $\overline{C_1}\cap \overline{C_2}$ is at most one  point.
 %$Y \cap E(X) \subset E(Y)$.

% There is a one to one map from $E(Y)$ to $E(X)$, in
%particular if $E(X)$ is countable then it is so for $E(Y)$.
\end{lem}
\begin{proof} Suppose that $\overline{C_1}\cap \overline{C_2}$
contains at least two distinct points $a,b$. Then by Theorem 10.10
of \cite{Nad} this intersection is connected, so it is arcwise
connected. Hence we obtain $[a,b]\subset \overline{C_1}\cap
\overline{C_2}$. There is four sequences $(a_n)_{n\geq 1}$,
$(b_n)_{n\geq 1}$ in $C_1$ and  $(a^{\prime}_n)_{n\geq 1}$,
$(b^{\prime}_n)_{n\geq 1}$ in $C_2$ such that $(a_n)_{n\geq 1}$,
$(a^{\prime}_n)_{n\geq 1}$ converges to $a$ and $(b_n)_{n\geq 1}$,
$(b^{\prime}_n)_{n\geq 1}$ converges to $b$.
 By Corollary \ref{J}, there is $\varepsilon>0$ and an arc $J\subset [a,b]$
 such that $J\subset [u,v]$ for any $u\in B(a,\varepsilon)$ and
 $v\in B(b,\varepsilon)$. There is $n>0$ such that $a_n, a^{\prime}_n \in B(a,\varepsilon)$ and $b_n, b^{\prime}_n \in B(b,\varepsilon)$.
 So we have $J \subset [a_n,b_n] \cap [a^{\prime}_n,b^{\prime}_n]$. Since $C_1$ and $C_2$ are arcwise connected then $[a_n,b_n]\subset C_1$ and
$[a^{\prime}_n,b^{\prime}_n]\subset C_2$ then $J\subset C_1 \cap
C_2$, absurd.

\end{proof}

\begin{lem} \label{hull}
Let $X$ be a dendrite and $F$ a  non empty closed subset of $X$. Let
$a\in F$, then
\begin{enumerate}
\item  $[F]=\cup_{z\in F }[a,z]$,
\item $E([F]) \subset F$ and we have $E([F])=F$ when $F\subset E(X)$.
\end{enumerate}
\end{lem}

\begin{proof}
%\begin{enumerate}
$(1)$ For any $z\in F $, $[a,z]$ is connected then the subset
$G:=\cup_{z\in F  }[a,z]$ is  connected and  contains $F$. We will
prove that $G$ is closed. Let $(w_n)_{n\geq 1}$ be a sequence in $G$
converging to $w \in X$. For any $n\geq 1$, there is $z_n \in F$
such that $w_n \in [a,z_n]$. Since $F$ is compact, by considering a
sub sequence of $(z_n)_{n\geq 1}$, we may assume that $(z_n)_{n\geq
1}$ converges to $z\in F$. It suffices to prove that $w\in [a,z]$.
 Suppose contrarily that $\delta:=d(w,[a,z])>0$ then $B(w,\frac{\delta}{2}) \cap B([a,z],
 \frac{\delta}{2})=\emptyset$. By Lemma \ref{dist}, since $(z_n)_{n\geq 1}$ converges to $z$ then
 there is $N>0$ such that diam$([z,z_n])<\frac{\delta}{2}$ when $n>N$, hence $[z,z_n]\subset
 B(z,\frac{\delta}{2})$ so $[a,z_n]\subset [a,z]\cup [z,z_n]\subset B([a,z],\frac{\delta}{2})$.
 Let $n>N$ such that $d(w,w_n)<\frac{\delta}{2}$. Then $w_n \in B([a,z],\frac{\delta}{2})\cap B(w,\frac{\delta}{2})$ hence $B([a,z],\frac{\delta}{2})\cap B(w,\frac{\delta}{2})\neq \emptyset$, absurd.
 So $w\in G$ then we conclude that $G$ is a subdendrite of
 $X$ satisfying $F \subset G \subset [F]$, hence $[F]=G$.\\
 $(2)$ Let $e\in E([F])$, by $(1)$ there is $z\in F$ such that $e\in [a,z]$.
 Since $1\leq ord(
 e,[a,z])\leq ord(e,[F])=1$ hence $e=z$ or $a$ then $e\in F$, so $E([F])\subset F$. Now, suppose that $F \subset E(X)$. It suffices to prove that $F \subset E([F])$. Since $F \subset [F]$ then $F \subset [F] \cap E(X)\subset E([F])$. This finish the proof of this Lemma.

%\end{enumerate}
\end{proof}

\begin{lem} $($\cite{MS}, Lemma $2.3)$ \label{cxe} Let $(C_n)_{n\geq1}$ be a sequence of pairwise disjoint
connected subsets of a dendrite $(X,d)$. Then we have $ \lim_{n \to
+\infty} \emph{diam}(C_n)=0 .$
\end{lem}

\begin{lem} (\cite{ACCS}, Theorem 3.3)\label{fo}
 If a dendrite has a closed set of endpoints then any point of it has finite order.
\end{lem}

\begin{thm} $($\cite{ACCS}, Theorem $3.2)$ \label{CH1}
Every subcontinuum of a dendrite with a closed set of endpoints is a
dendrite with a closed set of endpoints.
\end{thm}

\begin{prop}$($\cite{ACCS}, Corollary $3.5)$ \label{CH2}
If $X$ is a dendrite with $E(X)$ closed set then
$\overline{B(X)}\subset B(X)\cup E(X) .$ In particular, $B(X) \cup
E(X)$ is  closed.
\end{prop}

%\begin{cor}\cite{CH} \label{CH3}
%If $X$ is a dendrite with closed set of endpoints, then
%$\overline{B(X)} \subset B(X) \cup E(X)$.
%\end{cor}

\begin{prop} $($\cite{ACCS}, Corollary $3.6)$ \label{CH4}
If $X$ is a dendrite with $E(X)$ closed  then $B(X)$ is a discrete
set.
\end{prop}

\begin{lem} \label{Y<X}  $($\cite{CCP}, Proposition $4.14)$
If  $Y$ and $X$ dendrites with $Y \subset X$, then
$E(Y)^{\prime}\subset E(X)^{\prime}$.
\end{lem}

\begin{lem} \label{card} $($\cite{Ni}, Lemma $4)$
If $Y$ and $X$ are dendrites with $Y \subset X$, then $card(E(Y))
\leq card(E(X))$.
\end{lem}
For a subset $A$ of $X$, we denote by card$(A)$ the number of
elements of $A$. If $A=\emptyset$, we take card$(A)=0$.

\begin{defn}
Let $ f: X \to X $ be a dendrite map and $I, J$ two arcs in $X$. We
say that $I, J$ form an arc horseshoe for $f$ if  $f^n(I) \cap
f^m(J) \supset I \cup J$ for some $n, m \in \mathbb{N}$, where $I,
J$ have exactly a common one endpoint.
\end{defn}

\begin{thm}    \cite{BHS}
Let $X$ be a compact metric space and $f:~X\to X$ a continuous map.
If $h(f)>0$ then $f$ is Li-Yorke chaotic.
\end{thm}
When $X$ is a compact interval, it is well known that if $f^n$ has
an arc hoseshoe for some $n\in \mathbb{N}$, then $h(f)>0$. Actually,
for dendrite map, we have
\begin{thm} $($\cite{KKM}, Theorem $2)$
Let $f:~X \to X$ be a dendrite map. If $f$ has an arc horseshoe then
$h(f)>0$.
\end{thm}

%\begin{thm} \cite{Mak}
%A dendrite $X$ has the periodic point property if and only if the
%power of the set of its endpoints card $E(X)<c$, where $c$ is the
%power of the continuum.
%\end{thm}

\begin{lem} \label{WI} $($\cite{BC}, page 71$)$ Let $(X,d)$ be a compact metric space and $f:~X\to X$ a continuous map. Let $F$ be a proper closed subset of an $\omega$-limit set
$L=\omega_f (x)$, then
$$ \overline{f(L \backslash F)}  \cap F\neq \emptyset .$$
\end{lem}
\medskip

  Lemma \ref{WI} is equivalent to the following Lemma:
\begin{lem} \label{WIbis}
If $G$ is a non empty open subset of $L$ (relatively to $L$) and
such that $f(\overline{G})\subset G$ then $G=L$.
\end{lem}
\begin{proof}
Suppose contrarily that there is a non empty open subset $G$ of $L$
such that $f(\overline{G})\subset G$ and $G\neq L$. Then
$F:=L\backslash G$ is a non empty closed subset of $L$ and
$\overline{f(L\backslash F)}\cap F= \overline{f(G)}\cap
F=f(\overline{G})\cap F \subset G \cap F=\emptyset$, absurd.
\end{proof}

\section{\textbf{On $\omega$-limit set containing a periodic point.}}
To prove  Theorem A, we need the following Lemmas

\begin{lem} (\cite{CWZ}, Proposition $4.4$)  \label{de}
Let $X$ be a dendrite with $E(X)$ closed   and $E(X)^{\prime} \neq
E(X)$. Then for any $e\in E(X) \backslash E(X)^{\prime}$, there is
$b\in X, b\neq e$ such that $[e,b]$ is a neighborhood of $e$ and
$[e,b] \cap B(X)=\emptyset$.
\end{lem}

%Let $e\in E(X) \backslash D(E(X))$. Suppose that for any $b\in X
%\backslash \{e\}$, $[e,b]\cap B(X)\neq \emptyset$. Then there is an
%infinite sequence of branch points, $(b_n)_{n\geq 1}$ in $(e,b)$
%converging to $e$ for some $b\in X\backslash \{a\}$. For any $n\geq1$, let $e_n \in E(X)$ such that
%$(b_n,e_n] \cap (b,e)=\emptyset$. The sequence
%$((b_n,e_n))_{n\geq1}$ is pairwise disjoint the by Lemma \ref{cxe},
%$\lim_{n\to +\infty}diam((b_n,e_n))=0$ so $\lim_{n\to
%+\infty}d(b_n,e_n)=0$, since $\lim_{n\to +\infty}d(b_n,e)=0$ then
%$\lim_{n\to +\infty}d(e_n,e)=0$ this implies that $e\in D(E(X))$,
%absurd. Now, let $b\in X\backslash \{e\}$ such that $[b,e]\cap
%B(X)=\emptyset$, then $[e,b)$ is the connected component of $X
%\backslash \{b\}$ containing $e$, so $[e,b)$ is open hence $[a,b]$ is a neighborhood of $e$.

\begin{lem} \label{Le0}
Let $X$ be a dendrite with countable closed set of endpoints and let
$Y$ be a subdendrite of $X$. Denote by $Y_1:=Y \backslash [B(X) \cup
E(Y)]=\cup_{i\in I}J_i$, where  $(J_i)_{i\in I}$ is the sequence of
the  connected components of $Y_1$. Then $I$ is at most countable
and each $J_i$ is an open free arc in $X$.

\end{lem}
\begin{proof}   Set $F=(B(X)\cap Y)\cup E(Y)$. We will show that
 $Y_1=Y \backslash F$
is closed in $Y$: Indeed, by theorem \ref{CH1}, since $E(X)$ is
closed, so $E(Y)$ is also closed. Let $(b_n)_{n\geq1}$ be an
infinite sequence in $B(X)\cap Y$ converging to $b$. Then $b\in Y$
and by Proposition \ref{CH2}, $b\in B(X) \cup  E(X)$, hence $b\in
(Y\cap B(X)) \cup (Y \cap E(X))\subset (Y\cap B(X)) \cup E(Y)$ so
$F$ is closed in $Y$, hence $Y_1$ is open in $Y$.
 So by (\cite{Nad}, p. 120), each component $J_i$ is open in $Y$. For any $i\in I$,
 $J_i$ is open in $X$, since $J_i \cap B(X)=\emptyset$.

Let us prove that $I$ is at most countable. For any $i\in I$, write
$\overline{J_i}=[a_i,b_i]$ and define the map $h:~I \to (B(X)\cup
E(Y))^2 $
 as follow: $\forall i\in I, h(i)=\{a_i, b_i\}$. The map $h$ is well defined (since $X$ is uniquely arcwise
connected) and it is one-to-one then $I$ is at most countable since
$B(X) \cup E(Y)$ is at most countable.

\end{proof}

\begin{lem} \label{Le1}
Let $f:~X \to X$ be a dendrite map such that $E(X)$ is closed and
countable. Let $a\in Fix(f)$ and  $L:=\omega_f (x)$ an uncountable
$\omega$-limit set such that $L \cap P(f)=\emptyset$ then for any
$y\in L$, there is $p, k \geq0$ such that $[a,f^{k}(x)]\subset
[a,f^{p}(y)]$.
\end{lem}
\begin{proof}
 Let $y\in L$. We have $\omega_f(y)$ is a closed
  invariant subset by $f$  then there is a minimal subset, denoted by $K $, in
   $\omega_{f}(y)$. Since $L\cap P(f)=\emptyset$ then $K$ has no periodic point.
   So $K$ is
infinite and has no isolated point, hence $K$ is uncountable so it
is for $\omega_f (y)$. Now, denote by $(C_i)_{i\in\mathbb{N}}$ the
sequence of connected components of $X \backslash (B(X)\cup E(X))$.
By Lemma \ref{Le0}, each $C_i$ is an open free arc in $X$. There is
$j\in \mathbb{N}$ such that $\omega_f (y) \cap C_j$ is uncountable.
Let $u, v\in \omega_{f}(y) \cap C_j$ such that $u\in (a,v)$. There
is two open disjoint arcs $I_u, I_v$ in $C_j$ such that $u\in I_u, v
\in I_v$. Let $p, k^{\prime}\geq 0$ such that $f^{p}(y)\in I_v$ and
$f^{k^{\prime}}(y) \in I_u$, since $f^{k^{\prime}}(y)\in L$, there
is $k>0$ such that $f^k (x) \in I_u$, so we obtain the inclusion
$[a,f^{k}(x)] \subset [a,f^{p}(y)]$.
\end{proof}

\medskip

 \textit{ Proof of Theorem A.}

$(1)$ Denote by $L=\omega_f (x)$. Suppose that $ L\cap P(f)
\nsubseteq E(X)^{\prime} $, there is $a\in L\backslash
E(X)^{\prime}$ such that $f^{N}(a)=a$ for some $N>0.$ Since $L$ is
infinite and $\forall 0\leq i \leq N-1,
f(\omega_{f^N}(f^{i}(x)))=\omega_{f^N}(f^{i+1}(x))$ then
$\omega_{f^N}(f^{i}(x))$ is infinite for any $0\leq i \leq N-1$. Let
$0\leq j \leq N-1$ such that $a\in \omega_{f^N}(f^{j}(x))$, so we
may assume that $a \in Fix(f)$. By Corollary \ref{fo}, since $E(X)$
is closed then $1\leq n:=ord(a)<+\infty$. By Corollary \ref{CH4} and
Lemma \ref{de}, $a$ has a neighborhood, $V$, which is a tree such
that $V \cap B(X) \subset \{a\}$. We can write
$V=\cup_{i=1}^{n}[a,b_i]$ such that the subsets $(a,b_i]; 1\leq i
\leq n$ are pairwise disjoint.
\medskip

\textbf{Claim 1.} \emph{There is $1\leq i_0 \leq n$ and an infinite
sequence of periodic points in $(a,b_{i_0})$ converging to $a$.}\\

 By (\cite{BC}, Lemma 4) $a$
is not isolated relatively to $L$, then there is an infinite
sequence in $L$, say $(y_n)_{n\geq 1}$, converging to $a$. Since $V$
is a neighborhood of $a$ then we may assume that $(y_n)_{n\geq
1}\subset V$. Let $1\leq i_0 \leq n$ such that $a\in
\overline{(a,b_{i_0}] \cap (y_n)_{n\geq 1}}$. By considering a sub
sequence we may assume that $(y_n)_{n\geq 1}\subset (a,b_{i_0}]$.
Let $c\in (a,b_{i_0})$ arbitrarily. There is $1 \leq n_1<n_2<n_3$
such that $y_{n_2}\in (y_{n_1},y_{n_3}) \subset (a,c)$. Let $I_1,
I_2$ and $I_3$ a disjoint open arc in $(a,c)$ such that $y_{n_i}\in
I_i; i=1,2,3$. There is $n,m\geq 0$ and $p,q>0$ such that
$f^{n}(x)\in I_1, f^{m}(x) \in I_3$ and $f^{p}(f^{n}(x)),
f^{q}(f^{m}(x))\in I_2$. So we have $\{f^{p}(f^{n}(x)),
f^{q}(f^{m}(x))\} \subset (f^{n}(x),f^{m}(x))$, then by \cite{AKLS}
(one can use also Theorem 2.13 of \cite{MS}), $P(f) \cap
(f^{n}(x),f^{m}(x)) \neq \emptyset $ hence $P(f) \cap (a,c) \neq
\emptyset$. This finish the proof of Claim 1.
\\ Now, denote by $(C_k)_{1\leq k \leq n}$ the sequence of connected components of
 $X \backslash \{a\}$ such that  $b_k\in C_k, \forall 1 \leq k \leq
 n$ and  let $c\in Fix(f^r)\cap
(a,b_{i_0}); r\geq 1$ such that $L \cap (C_{i_0} \backslash
[a,c])\neq \emptyset $. Denote by $g=f^r$. There is $n\geq 0$ such
that $f^{n}(x) \in (a,c)$, since $L=\omega_f (f^{n}(x))$, we may
assume that $n=0$. We distinguish two cases both of them lead us
to a contradiction:\\
\medskip

\textbf{Case 1.} \emph{If} $O_g (x) \nsubseteq C_{i_0}$. Let $1\leq
j \leq n; j\neq i_0$ and $p, k>0$ such that $g^{p}(x) \in C_j$ and
$f^{k}(x)\in C_{i_0} \backslash [a,c]$. Let $I=[a,x]$ and $J=[x,c]$.
Then we have $f^{k}(I)\supset [a,f^{k}(x)]\supset [a,c]=I \cup J$
and $f^{rp}(J)=g^{p}(J)\supset [c,g^{p}(x)] \supset [a,c]=I \cup J$,
so $I, J$ form an arc horseshoe then by \cite{KKM} and \cite{BGKM}
we have
$h(f)>0$, absurd. \\
\medskip

\textbf{Case 2.} \emph{If} $O_g (x) \subseteq C_{i_0}$. Denote by
$F_c= \cup_{n=0}^{+\infty}g^{-n}(c) \cap [a,c]$.
\\
$(a)$ If there is $z\in F_c \backslash \{c\}$ such that  $ (c,z]
\cap O_g (x) \neq \emptyset.$ Let $n\geq 0, k>0$ such that $g^{n}(x)
\in (c,z) $ and $g^{k}(z)=c$. Let $I=[c,g^{n}(x)]$ and
$J=[z,g^{n}(x)]$. Since $a\in \omega_g (x)$ then there is $p>k$ such
that $I \cup J \subset [c,g^{n+p}(x)]$. So we have $g^{p}(I)\supset
[c,g^{n+p}(x)]\supset I\cup J$ and $g^{k}(J) \supset
[c,g^{n+k}(x)]\supset I \cup J$, hence $g^{p}(J)=g^{p-k}(g^{k}(J))
\supset [c,g^{n+p}(x)]\supset I \cup J$ so $I, J$ form an arc
horseshoe for $g$,
hence $h(f)=\frac{1}{r}h(g)>0$, absurd. \\

$(b)$ If for any $z\in F_c; [c,z] \cap O_g (x) = \emptyset$. \\

\textbf{Claim 2.} \emph{We
have $O_g (x) \subset [a,c]$.}\\
  Since otherwise, there is $n>0$ such that
$g^{n}(x) \notin  [a,c]$. $(i)$ If $g^n(x) \in C_{i_0}\backslash
[a,c]$ then $g^{n}([a,x])\supset [a,g^{n}(x)]\ni c$, so there is
$c_{-1}\in (a,x); g^{n}(c_{-1})=c$,
 hence $x\in [c,c_{-1}] \cap O_g(x)$, a contradiction.
 $(ii)$ If $g^n(x) \in C_j$ for some $j\neq i_0$. Then $g^n ([c,x])\supset [c,g^n(x)]\ni a$. Let $a_{-1}\in (c,x); g^n(a_{-1})=a$ and $k>0$ such that $f^k (x) \in C_{i_0}\backslash [a,c]$. Then $f^{rn}([c,a_{-1}])=g^n ([c,a_{-1}])\supset [a,c]$ and $f^{k}([a,a_{-1}])\supset f^k ([a,x])\supset [a,f^k (x)]\supset [a,c]$, so $[a,a_{-1}]$ and $[c,a_{-1}]$ form an arc horseshoe for $f$, absurd. This finish the proof of Claim 2.\\
 \medskip

  Now, we have $\omega_g (x) \subset [a,c]$. Denote by
$[a,d]=[\omega_{g}(x)]$ the convex hull of $\omega_g (x)$. We remark
that $d\in \omega_g (x)$ since $\omega_g (x)$ is closed.\\
\medskip

  \textbf{Claim 3.} For any $n\geq 0$, $g^{n}([a,d))\subset [a,c)$.\\

 Suppose contrarily that there is $n>0$ such that $g^{n}([a,d))\nsubseteq
[a,c)$. Let $z\in (a,d), g^{n}(z) \notin [a,c)$. $(i)$ If $g^{n}(z)
\in C_{i_0}\backslash [a,c)$ then $c$ has an antecedent $c_{-1}$ by
$g^n$ in $(a,z]$, since $d\in [c,c_{-1})\cap \omega_g (x)$ then
$O_g(x)\cap (c,c_{-1})\neq \emptyset$, a contradiction. $(ii)$ If
$g^{n}(z) \in C_j$ for some $j \neq i_0$.  Then $g^n([c,z])\supset
[c,g^n(z)]\ni a$ then there is $a_{-1}\in (c,z]$ such that $g^n
(a_{-1})=a$. Let $k,p>0$ such that $f^{k}(x)\in (a,a_{-1}),
f^{k+p}(x)\in C_{i_0}\backslash [a,c]$. We have
$f^{p}([a,f^{k}(x)])\supset [a,f^{k+p}(x)]\supset [a,f^{k}(x)] \cup
[f^{k}(x),c]$ and $f^{nr}([c,f^{k}(x)])\supset
f^{nr}([c,a_{-1}])=g^{n}([c,a_{-1}])\supset [c,a]= [c,f^{k}(x)] \cup
[f^{k}(x),a]$, hence $[c,f^{k}(x)], [a,f^{k}(x)]$ form an arc
horseshoe for $f$, absurd. This finish the proof of the Claim 3.

Now, we have $ g^{n}([a,d])=g^{n}([a,d))\cup g^{n}(d) \subset
[a,c)\cup \omega_g (x) \subset [a,c]$. In the other hand, since
$[a,d]\supset \omega_g (x)$ so $g([a,d])\supset [a,d]$. The subset
$I=\cup_{n=0}^{+\infty}g^{n}([a,d])$ is connected included in
$[a,c]$ and strongly invariant by $g$  so it is for $J=\overline{I}
\subset [a,c]$. Now $g_{|J}:~J \to J$ is a continuous interval map
with $\omega_{g_{|J}}(x)=\omega_g (x); x\in J$ is infinite
containing a fixed point $a$, so by \cite{Sar}, $h(g)\geq
h(g_{|J})>0$ hence $h(f)>0$, absurd. We conclude that $L \cap P(f)
\subset E(X)^{\prime}$.
\medskip

$(2)$ Now, suppose that $L=\omega_f(x)$  uncountable and $E(X)$
countable. By $(1)$ of this Theorem, we have $L\cap P(f) \subset
E(X)^{\prime}$. Suppose that there is $a\in L \cap P(f)$, we may
assume that $a\in Fix(f)$. By Lemma \ref{Le0}, write $X\backslash
[B(X)\cup E(X)]=\cup_{i=1}^{+\infty}J_i$ where each $J_i$ is an open
free arc in $X$. There is $i_0>0$ such that $L \cap J_{i_0}$ is
uncountable. Write $J_{i_0}=(u,v)$ such that $v\in (a,u)$. There
exists $c\in Fix(f^r); r\geq 1$ such that $L \cap (u,c)$ is
uncountable. Denote by $g=f^r$. There is $0\leq i \leq r-1$ such
that $\omega_{g}(x_i)\cap (u,c)$ is uncountable; $ x_i=f^i (x)$.
There is $k\geq0, p>0$ such that $g^{k}(x_i), g^{k+p}(x_i) \in
(w,c)\cap (u,c)$ for some $w\in \omega_{g}(x_i)$. Denote by $X_0,
X_1$ the connected components of $X \backslash \{c\}$ such that
$a\in X_1$ and let $l_j= X_j \cap \omega_{g}(x_i); j=0,1$. By $(1)$,
$c\notin \omega_g(x_i)$ then $l_0$ and $l_1$ are two non empty
clopen sets relatively to $\omega_{g}(x_i)$, hence by Lemma
\ref{WIbis} we have $\forall n\geq1, g^{n}(l_0)\nsubseteq l_0$. Let
$y\in l_0$ such that $g^{p}(y) \in l_1$.

We will build an arc horseshoe. We
distinguish two cases:\\
\medskip

\textbf{Case 1.} \emph{There is an infinite sequence $(n_l)_{l>0}$ }
\emph{such that} $(g^{n_l}(x_i))_{l>0}$ \emph{converges to} $y$
\emph{and}  $g^{n_l}(x_i)\in (c,y); \forall l>0$. By continuity of
$g^p$, there is $n>0$ such that $g^n (x_i) \in X_0, g^{n+p}(x_i) \in
X_1$.
 Denote by
$I=[c,g^{k}(x_i)], J=[g^{k}(x_i),g^{n}(x_i)]$ if $g^{k}(x_i) \in
(c,g^{n}(x_i))$, (resp. $I=[c,g^{n}(x_i)],
J=[g^{n}(x_i),g^{k}(x_i)]$ if $g^{n}(x_i)\in (c,g^{k}(x_i))$). We
have $g^{p}(I) \cap g^{p}(J)\supset [c,g^{k+p}(x_i)]$ (respectively,
$ g^{p}(I) \cap g^{p}(J)\supset [c,g^{n+p} (x_i)]$). There is $s>0$
such that $[c,g^{k+p+s}(x_i)]\supset I \cup J$ hence $g^{p+s}(I)
\cap g^{p+s}(J) \supset I \cup J$ (resp. since $g^k(x_i)\in (c,w);
w\in \omega_g(x_i)$ then there is $r>0$ such that
$[c,g^{n+p+r}(x_i)]\supset I \cup J$ hence $g^{p+r}(I) \cap
g^{p+r}(J)\supset I \cup J$). So $I, J$ form an arc horseshoe for
$f$, absurd.
\medskip

\textbf{Case 2.} \emph{There is an  infinite sequence $(n_l)_{l>0}$
}  \emph{such that} $(g^{n_l}(x_i))_{l>0}$\emph{ converges to} $y$
\emph{and}  $y\in (c,g^{n_l}(x_i)); \forall l>0$.
 Similarly as in Case 1, we build an arc horseshoe by theses three points $c, g^{k}(x_i), y$ (resp. by $c,g^{n}(x_i),g^{k}(x_i)$ for a convenient integer $n$) if $g^{k}(x_i)\in (c,y))$ (resp. if $y\in (c,g^{k}(x_i))$).

%Then there is
%$n>0$ such that $g^n (x_i) \in X_0$ and $g^{p}(g^{n}(x_i))\in X_1$.
%Denote by $I=[c,g^n (x_i)], J=[g^n (x_i), g^k (x_i)]$ if $g^n (x_i)
%\in (c,g^k(x_i))$ ( respectively, $I=[c,g^k (x_i)], J=[g^k (x_i), g^n
%(x_i)]$ if $g^k (x_i) \in (c,g^n(x_i))$). Since
%$g^{p}(J)\supset[g^{k+p} (x_i), g^{n+p} (x_i)]\ni c$ then we have
%$g^{p}(I)\cap g^{p}(J) \supset [c,g^{n+p}(x_i)]$ (respectively,
%$g^{p}(I)\cap g^{p}(J) \supset [c,g^{k+p}(x_i)]$). There is
%$s>|k-n|$ such that $[c,g^{n+p+s}(x_i)]\supset I \cup J$ so we have
%$g^{p+s}(I)\cap g^{p+s}(J) \supset [c,g^{n+p+s}(x_i)]\supset I \cup
%J$, (respectively, $g^{p+s+n-k}(I)\cap g^{p+s+n-k}(J) \supset g^{s+n-k}([c,g^{k+p}(x_i)])\supset [c,g^{p+s+n}(x_i)]\supset I \cup
%J$) then $I, J$ form an arc horseshoe, absurd.

This finish the proof of Theorem A. \hfill $\square$

\section{\textbf{Examples of dendrite maps with zero topological entropy.}}
\subsection{Example $1$.} We build a dendrite $X$ with $E(X)$ countable closed set and a map $f:~X\to X$ with zero topological entropy     having an infinite $\omega$-limit set containing a periodic point.\\
\medskip

\textbf{Construction of the dendrite $X$.} For any $n\geq0$, let
\begin{enumerate}
\item $a_n=(1-\frac{1}{n+1},0), b_n=(1-\frac{1}{n+1},\frac{(-1)^{n}}{n+1})$
and $e=(1,0)$,
\item $w_n\in (a_n,b_n)$,
\item $(b_n^k)_{k\geq 0}$ be a monotone sequence in $[b_n,w_n)$
converging to $w_n$ where $b_n^0=b_n$,
\item $X=\cup_{i=0}^{+\infty}[a_i,b_i] \cup [a_0,e]$.
\end{enumerate}
We can see that $X$ is a dendrite with $E(X)=\{b_n; n\geq 0\}
\cup\{e\}$  closed and $E(X)^{\prime}=\{e\}$.\\
\medskip

 \textbf{Construction of the  map $f:~X \to X$.} (See Figure $1$)
 \quad We define $f$ as follows: For any $n, k\in \mathbb{Z}_+$,
 \begin{enumerate}
 \item $f(e)=e$, \\ $f$ maps linearly
 \item  $[a_{n},a_{n+1}]$ to
 $[a_{n+1},a_{n+2}]$ such that $f(a_{n})=a_{n+1}$,
 \item  $[a_{2n+2},w_{2n+2}]$ to
 $[a_{2n+3},w_{2n}]$,
 \item  $[a_{2n+1},w_{2n+1}]$ to
 $[a_{2n+2},w_{2n+3}]$,
 \item  $[a_0,w_0]$ to $[a_1,w_1]$,
 \item  $[w_{2n+2},b_{2n+2}] $ to $[w_{2n},b_{2n}]$ such that
 $f(b_{2n+2}^k)=b_{2n}^{k+1}$,
 \item  $[w_0,b_0]$ to $ [w_1,b_1]$ such that $f(b_0^k)=b_1^k$,
 \item  $[w_{2n+1},b_{2n+1}^1]$ to $[w_{2n+3},b_{2n+3}]$ such that
 $f(b_{2n+1}^{k+1})=b_{2n+3}^{k}$,
 \item  $[b_{2n+1},b_{2n+1}^1]$ to $ [b_{2n+3},b_{2n+2}]$.
  \end{enumerate}
  \begin{figure}[H]
\begin{center}
\includegraphics[width=14cm,height=8cm]{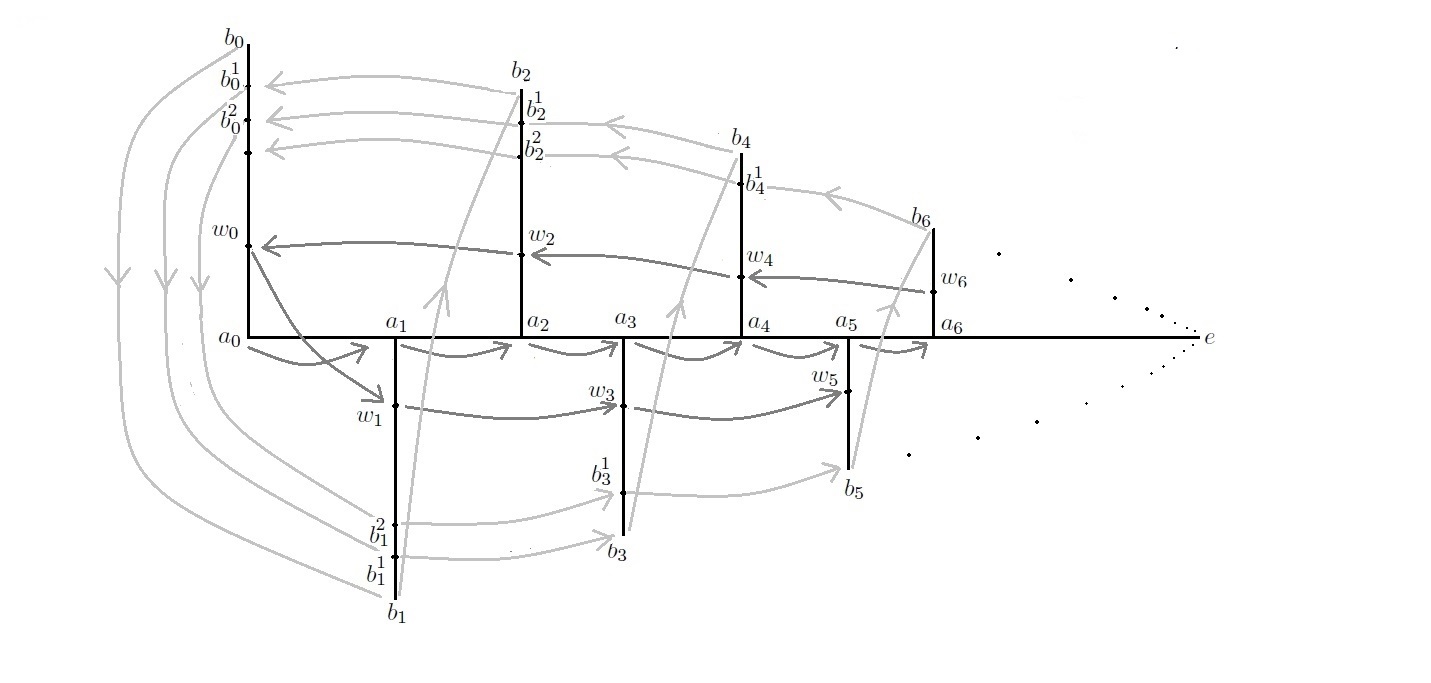}
\caption{Dendrite with $E(X)^{\prime}$ is reduced to one point.}
\end{center}
\end{figure}
   Denote by $x=b_0$ and for any $n\geq 0, x_n=f^{n}(x)$. Then the first fifteen elements of the orbit of $x$ are:

   $ x=b_0,  x_1=b_1, x_2=b_2, x_3=b_0^1, x_4=b_1^1, x_5=b_3,
   x_6=b_4, x_7=b_2^1, x_8=b_0^2, x_9=b_1^2, x_{10}=b_3^1,
   x_{11}=b_5, x_{12}=b_6, x_{13}=b_4^1, x_{14}=b_2^2, x_{15}=b_0^3 .$

\begin{lem} The map $f:~X \to X$ obtained is continuous and satisfy the
following properties:
\begin{enumerate}
\item $\omega_f (b_0) =\{w_n; n\geq0 \} \cup\{e\}$. Therefore  $(b_0,e)$ is a
Li-Yorke pair,
\item for any $y\in X$, we have either $\omega_f(y)=\{e\}$ or
$\omega_f(y)=\omega_f(b_0)$,
\item $f$ has zero topological entropy,
%\item $f$ is Li-Yorke chaotic.
\item $(X,f)$ is  proximal i.e any pair $(x,y)\in X ^2$ is proximal.
\end{enumerate}
\end{lem}
\begin{proof}
It is easy to prove $(1)$ and $(2)$. Let prove $(3)$. Denote by
$R(f)$ the set of recurrent points of $f$ i.e $R(f):=\{x\in X, x\in
\omega_f(x) \}$. Since $E(X)$ is countable then by \cite{MS} we have
$\overline{R(f)}=\overline{P(f)}$. Since $P(f)=\{e \}$ then
$\overline{R(f)}=\{e\}$. Now by \cite{BC}, page 196 we have
$h(f)=h(f_{|\overline{R(f)}})=h(f_{|\{e \}})=0$.\\
$(4)$ Let prove that $(X,f)$ is proximal. Let $(x, y) \in X^2$. Fix
$\varepsilon>0$. There is $L>0$ such that for any $i\geq 0$,
$\{f^i(y),f^{i+1}(y),\dots,f^{i+L}(y)  \} \cap
B(e,\frac{\varepsilon}{2})\neq \emptyset$.
 Since $e=f(e)\in \omega_f(x)$ %and by continuity of $f, f^2,f\dots, f^{L}$,
there is $p>0$ such that $\{  f^p(x),f^{p+1}(x),\dots,f^{p+L}(x)\}
\subset B(e,\frac{\varepsilon}{2})$. Let $0\leq k \leq L$ such that
$f^{p+k}(y)\in
B(e,\frac{\varepsilon}{2})$, we obtain $d(f^{p+k}(x),f^{p+k}(y))<\varepsilon$.\\
\end{proof}
\medskip
\subsection{Example $2$.}
We will prove that there is  a non chaotic dendrite map $f:~X \to X$
with $E(X)$ countable but non closed set having an uncountable
$\omega$-limit set containing a periodic point.
\medskip

\textbf{ Construction of the dendrite $X$.}\\
  For any $n\geq 1$, denote by:
 \begin{itemize}
 \item $A=(0,0)$ and $B=(1,0)$,
 \item   $S_n=\{\frac{i}{2^n} \text{ where } 1\leq i \leq 2^n \text{ is odd } \}$
 \item $S_{2n+1}=\{a_{2^{2n}}>a_{2^{2n}+1}>\dots>a_{2^{2n+1}-1}
 \}; n\geq 0$ and $S_{2n}=\{a_{2^{2n-1}}<a_{2^{2n-1}+1}<\dots<a_{2^{2n}-1}
 \}$,
 \item $I_k=[A_k,B_k]$ where $A_k=(a_k,0)$ and $B_k=(a_k,\frac{1}{n+1})$ for any
 $n\geq 0$ and $k\in\{2^{n},2^{n}+1,\dots,2^{n+1}-1\}$.
\end{itemize}
 So the set $X:=[A,B]\cup (\cup_{k\geq1} I_k)$ is a dendrite with
$E(X)=\{B_k; k\geq 1 \}$ and $\overline{E(X)}=E(X)\cup [A,B]$.
\medskip

 \textbf{ Construction of the  map $f:~X \to X$.}( See Figure $2$).\\
 The map $f$ is defined as follow: $f$ fix any point in $[A,B]$ and
 for any $k\geq 1$, $f$ maps linearly $I_k$ to $I_{k+1}$ such that the center of $I_k$ is sent to
 $A_{k+1}$.

 \medskip

 The map $f:~X \to X$  satisfies the
following properties:  for any $k\geq 0, \omega_f(B_k)=[A,B]$ and
for any $y\in X \backslash E(X)$, there is $p\geq0$ such that $f^p
(y)\in [A,B]$. It follows that $(u,v) \in X^2$ is a Li-Yorke pair if
and only if either $(u,v)$ or $(v,u) $ lies to $ E(X) \times X
\backslash E(X)$. Then $f$ is not chaotic so $h(f)=0$ but
$\omega_f(B_0)=[A,B]=Fix(f)$.

\begin{figure}[H]
\begin{center}
\includegraphics[width=14cm,height=6cm]{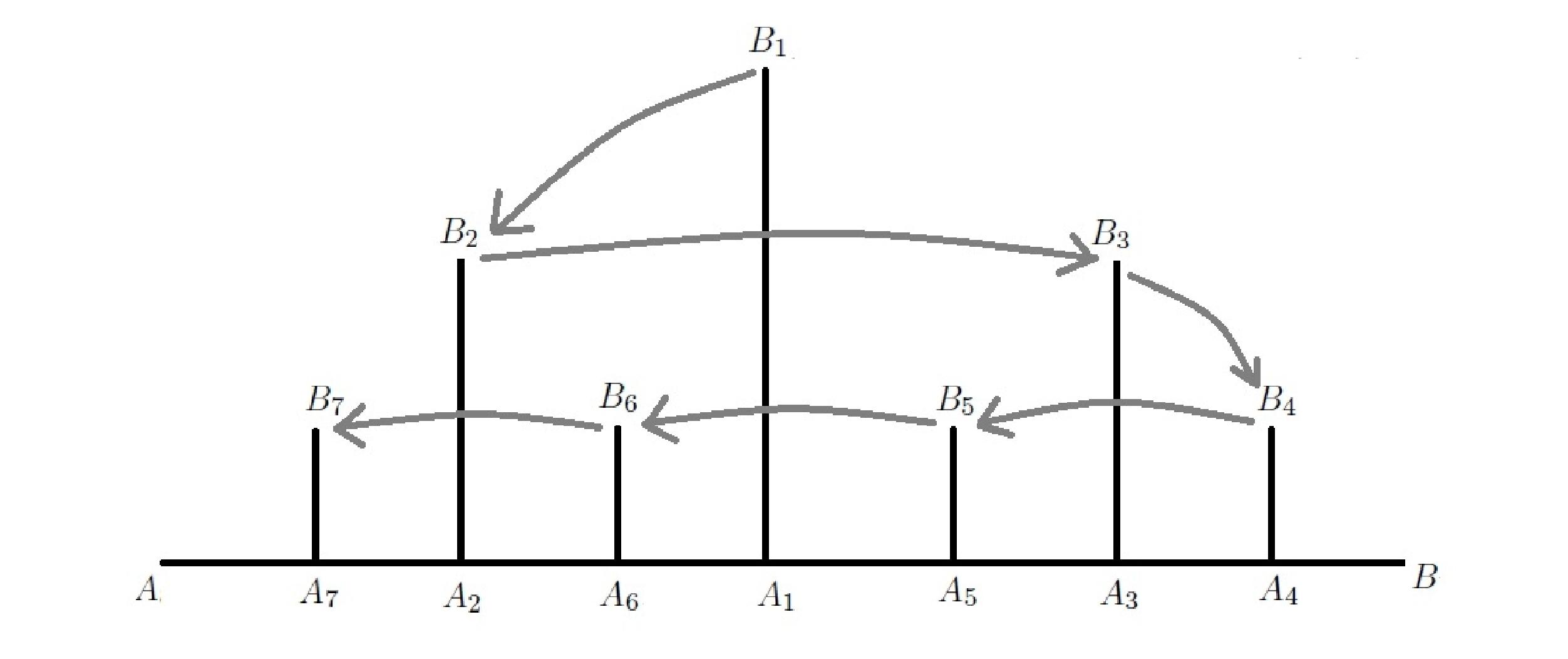}
\caption{Dendrite with a non closed countable  set of endpoints.}
\end{center}
\end{figure}

\subsection{Example $3$.}
\begin{lem} \cite{FW} \label{sym}
There is $s\in \Sigma_2:=\{0,1\}^{\mathbb{Z}_+}$ a recurrent point
such that $\omega_{\sigma}(s)$ is an uncountable $\omega$-limit set
containing a fixed point and $\sigma_{| \omega_{\sigma}(s)}$ has
zero topological entropy, where $\sigma=\Sigma_2 \to \Sigma_2$ is
the shift map defined as follow: $\forall x=(x_n)_{n\geq0}\in
\Sigma_2, \sigma(x)=(x_{n+1})_{n\geq0}$.
    \end{lem}
\medskip

\begin{lem} $($\cite{ACCS}, Proposition $6.8$, p. $16)$ \label{unc}
Each dendrite with an uncountable set of its endpoints contains a
homeomorphic copy of the Gehman dendrite.
\end{lem}

\begin{prop}
Let $X$ be a  dendrite with $E(X)$ uncountable. Then there is a
continuous map $f:~X \to X$ with zero topological entropy having an
uncountable $\omega$-limit set containing a periodic point.
\end{prop}

\begin{proof} Let $X$ be a dendrite with $E(X)$ uncountable.
If $X$ is a Gehman dendrite.
 The set $E:=E(X)$ is homeomorphic to $\omega_{\sigma}(s)$ where $s$ is
  defined in the Lemma \ref{sym}. Let $f$ be the map defined in \cite{Koc}
  with the same notations such that $f$ act in $E$ as the subshift
  $\sigma_1:=\sigma_{|\omega_{\sigma}(s)}$. We may assume that $E=\omega_{\sigma}(s)$. Any
point $y\in X \backslash E$ is eventually mapped to $c$ so
$\overline{R(f)}\subset E \cup\{c\}$. Since
$h(f)=h(f_{|\overline{R(f)}})$ then $h(f)\leq
max(h(f_{|E}),h(f_{|c}))=h(f_{|E})=h(\sigma_1)=0$ hence $h(f)=0$.
Also we have $\omega_f(s)=\omega_{\sigma}(s)$ is uncountable
containing a periodic point. Generally, by Lemma \ref{unc}, $X$
contains a homeomorphic copy of the Gehman dendrite, $G$. Denote by
 $r_G:~X \to G$  the retraction map. Let $f:~G \to G$ defined above and we set
$g=f\circ r_G:~X \to X$. Then we have $g$ is a dendrite map on $X$,
$h(g)=h(f)$ and there is $x\in G$ such that
$\omega_g(x)=\omega_f(x)$ is uncountable containing a periodic
point. This finish the proof of the Proposition.
\end{proof}

\bigskip

 \section{\textbf{ Decomposition of an uncountable $\omega$-limit
 set.}}

\begin{Ass} \label{L1}
  We assume that  $X$ is a dendrite with $E(X)$  closed and
  $E(X)^{\prime}$  finite.  Let $f:~X\to X$  be a dendrite map
  with zero topological entropy having an uncountable omega limit set
  $L:=\omega_{f}(x)$ $($such map exists by $($\cite{Sm},Theorem 2.7$))$. Denote by
  $M=[L]$ the convex hull of $L$.

  \end{Ass}
\medskip
The aim of this paragraph is to prove the following proposition.
\begin{prop} \label{prop}
Let Assumption \ref{L1} be satisfied. Then
\item there is an integer $n\geq 2$ an a connected subset $S$ of $X$
such that:\\
$(i)$ $S, f(S),\dots, f^{n-1}(S)$ are pairwise disjoint,\\
$(ii)$ $f^{n}(S)=S$,\\
$(iii)$ $L \subset \cup_{i=0}^{n-1}f^{i}(S)$,\\
$(iv)$ $\forall 0\leq i \leq n-1, f( L\cap f^{i}(S))=L\cap
f^{i+1}(S)$.

\end{prop}
\medskip

\begin{rem}
In fact, $(iv)$ is an immediate consequence of $(i),(ii)$ and
$(iii)$.

\end{rem}

\begin{rem}
%\item $M$ is a subdendrite of $X$ such that $E(M)$ is closed
%countable subset  and $E(M)^{\prime} \subset E(X)^{\prime}$, See \cite{CWZ},
 If $f:~X \to X$ is a tree map with zero topological entropy,
 then any infinite $\omega$-limit set is uncountable.
\end{rem}
\medskip

\begin{lem} \label{YD}
Let assumption \ref{L1} be satisfied. Let $Y$ be a (non degenerate)
subdendrite of $X$ such that $E(Y) \cap E(X)^{\prime}=\emptyset$,
then $Y$ is a tree and $X \backslash Y$ has finitely many connected
components. Furthermore, there is a pairwise disjoint subdendrites
$D_1,D_2,\dots, D_n$ in $X$ such that $X \backslash Y \subset
\cup_{i=1}^{n}D_i; n>0$ and for any $1\leq i \leq n, D_i \cap Y$ is
reduced to one point.
\end{lem}
\medskip

\begin{proof}
Suppose contrarily that $X \backslash Y$ has infinitely many
connected components denoted by $(C_n)_{n\geq 1}$. For any $n\geq
1$, let $e_n \in E(\overline{C_n})\cap Y$ and  $a_n \in
E(\overline{C_n})\backslash \{e_n\} \subset E(X)$. For any $n\geq 1,
(e_n,a_n] \subset C_n$ so $((e_n,a_n])_{n\geq 1}$ is a pairwise
disjoint connected subsets in $X$. By Lemma \ref{cxe}, $\lim_{n \to
+\infty}$diam$((e_n,a_n])=0$ hence $\lim_{n\to
+\infty}d(a_n,e_n)=0$. Since  $E(X)$ is  closed, we may assume that
the sequence $(a_n)_{n\geq 1}$ converges to a point $a\in
E(X)^{\prime} $  hence $(e_n)_{n\geq 1}$ converges to $a\in E(Y)$.
Since $E(Y)$ is closed then $a\in E(Y) \cap E(X)^{\prime}$, absurd.
Let prove that $Y$ is a tree. By Lemma \ref{Y<X}, we have
$E(Y)^{\prime}\subset E(X)^{\prime}\cap E(Y)$ then
$E(Y)^{\prime}=\emptyset$, since $E(Y)$ is closed then $E(Y)$ is
finite. Hence $Y$ is a tree.
 Now, denote by $C_1, C_2,\dots, C_k; k>0$ the connected components
 of $X \backslash Y$. For any $1\leq n \leq k$; denote by $\{e_n\}=Y \cap
 \overline{C_n}$. For any $ n,m \in \{1,2,\dots,k\}$, $n \sim m\Leftrightarrow
 e_n=e_m$. For any $1\leq n \leq k$, let $D_n:=\cup_{m\sim
 n}\overline{C_m}$. Then we obtain a pairwise disjoint subdendrites $D_1,D_2, \dots, D_s; s\geq 1$ satisfying the conditions of the
 Lemma.
\end{proof}
\medskip

\emph{Proof of Proposition \ref{prop}.}
We distinguish two cases.\\
\medskip

\textbf{Case 1:  $M\cap Fix(f)\neq \emptyset$.}

 Let $a\in Fix(f)\cap M$. Denote by
$F_a =\cup_{n=0}^{+\infty}f^{-n}(a) \cap M$ and
$Y:=[\overline{F_a}]$.
 We remark that by Theorem A, $a\in M\backslash L \subset M\backslash E(M)$, hence
$2\leq ord(a,M)$.

\begin{lem} \label{L3}
Let Assumption \ref{L1} be satisfied. Then $L \subset M \backslash
Y$.
\end{lem}

\begin{proof}
\textbf{Claim 1.} For any  $z\in F_a$, we have $[a,z]\cap L=\emptyset$.\\

 Suppose contrary that $[a,z] \cap L \neq \emptyset$ for some
$z\in f^{-p}(a) \cap M; p\geq 1$.  Let $y$ be such a point in the
intersection above.
 Denote by $I=[a,y], J=[y,z]$ then we have $
f^{p}(I)\cap f^{p}(J)\supset [a,f^{p}(y)]$. By  Lemma \ref{Le1},
there is $i, j \geq 0$ such that $ [a,f^{p+j}(y)] \supset
[a,f^{i}(x)]$ so we have $f^{p+j}(I) \cap f^{p+j}(J) \supset
[a,f^{i}(x)]$. Let $w\in L$ such that $z \in (y,w)$ then there is
$r\geq0$ such that $[a,z] \subset [a,f^{i+r}(x)]$. Then
$f^{p+j+r}(I)\cap f^{p+j+r}(J)\supset [a,f^{i+r}(x)] \supset [a,z]=I
\cup J$. Then  $I,J$ form an arc horseshoe for $f$, so $h(f)>0$,
absurd. This finish the proof Claim $(1)$.

\bigskip

\textbf{Claim 2.} $L\cap \overline{F_{a}}=\emptyset$.\\

\item  Suppose that there is $y \in L \cap \overline{F_a}$. First, we
will prove that $f^{n}(y) \notin \overline{F_a}$ for some $n\geq 0$.
Suppose contrary that $O_{f}(y) \subset \overline{F_a}$, then we
have $\omega_{f}(y)\subset \overline{F_a}$, since
$\omega_{f}(y)\subset L$ then it contain no periodic point, so
$\omega_f (y)$ is uncountable. By Lemma \ref{Le0}, there is an open
free arc $I$ in $X$ such that $card(\omega_f (y) \cap I)> 2$. Let
$z_1, z_2 $ two distinct points in $\omega_f (y) \cap I$ such that
$z_1 \in (a,z_2)$ and $I_1, I_2$ two open disjoint arc in $I$
containing $z_1, z_2$, respectively.
 Then there is $k, p\geq 0$ such that $f^{k}(y) \in I_1$ and $f^{p}(y)\in I_2$.
 Since $I_2$ is open
 containing $f^p (y) \in \overline{F_a}$, then there is a point  $z\in F_a \cap I_2$,
 hence $f^{k}(y)\in L \cap[a,z]$,
 this contradict Claim $(1)$.
  Now, fix $n_0\geq 1$ such that $f^{n_0}(y)
\notin \overline{F_a}$. Since $X \backslash \overline{F_a}$ is a non
empty open subset in $X$ and since $f^{n_0}$ is continuous, there is
an open subset $U$ of $X$ containing $y$ such that $f^{n_0}(U)
\subset X \backslash \overline{F_a}$, let $t\in F_a \cap U$, then
$f^{n_0}(t) \in f^{n_0}(U)\subset X \backslash \overline{F_a}$, but
$f^{n_0}(t) \in F_a$, absurd. This finish the proof of Claim $(2)$.

By  Claims $(1)$ and $(2)$, we have
 for any $z\in \overline{F_a}$, $[a,z] \cap L=\emptyset$,
 so we have $L\subset M \backslash \cup_{z\in
\overline{F_a}}[a,z]$, hence by Lemma \ref{hull} we obtain $L
\subset M \backslash Y$. Since $E(Y)\subset \overline{F_a}$,
$E(M)^{\prime}\subset E(M) \subset L$ then $E(Y)\cap
E(M)^{\prime}=\emptyset$
    then by Lemma \ref{YD}, $Y$
    is a tree and  $M \backslash
    Y$ has finitely many connected components.
\end{proof}

\begin{lem} \label{C}
Let Assumption \ref{L1} be satisfied. In $(1)$ we suppose that $F_a
\neq \{a\}$, then
\begin{enumerate}
\item for any $y\in L, (a,y] \cap F_a \neq \emptyset$,
\item $M\backslash F_a$ has finitely many connected components intersecting $L$, we denote it by $C_1, C_2,
    \dots, C_n;$ $ n\geq 2$ and for any $k\in \{1,2,\dots,n\}$,
    we denote by  $l_k= L \cap C_k$,
\item for any $1\leq k \leq n$, there is a unique $j:=\sigma(k)\in \{1,2,\dots,n\}$
such that $f(C_k) \cap C_j \neq \emptyset$,
\item for any $1\leq k \leq n$, $f(l_k)=l_{\sigma(k)}$,
\item for any $1\leq k \leq n$, $l_k$ is a clopen set relatively to
$L$,
\item $\sigma$ is an $n$-cycle.

\end{enumerate}
\end{lem}

\begin{proof}

$(1)$ let $y\in L$ arbitrarily, $z\in F_a, z\neq a$ and $w\in L$
such that $z\in (a,w)$. By Lemma \ref{Le1}, there is $p, k\geq 0$
such that $[a,f^{k}(x)]\subset [a,f^p (y)]\subset f^p ([a,y])$. Let
$i\geq 0$ be such that $[a,z]\subset [a,f^{k+i}(x)]$ then
$[a,z]\subset f^{i}([a,f^{k}(x)])\subset f^{p+i}([a,y])$ so $z\in
f^{p+i}([a,y])$ hence there is $z_{-1}\in [a,y]$ such that
$f^{p+i}(z_{-1})=z\in F_a$ with $z_{-1}\neq a$.   So $ z_{-1}\in (a,y] \cap F_a$.\\

$(2)$ If $F_a \neq \{a\}$.  We have $F_a \subset Y$ so a connected
subset of $M\backslash Y$ is also a connected subset of $M\backslash
F_a$, since $L \subset M \backslash Y$ and by Lemma \ref{YD}, $M
\backslash Y$ has finitely many connected components so the
connected subsets of $M \backslash F_a$ intersecting $L$ are finite.
If $F_a =\{a\}$, since $ord(a,M)<+\infty$ then $M\backslash \{a\} $
has finitely many connected components, (exactly $n:=ord(a,M)$
components), and $L
\subset M \backslash \{a\}$.\\

$(3)$ Let $1\leq k \leq n$, since $f(l_k) \subset f(C_k)$  and
$f(l_k)\subset f(L)=\cup_{i=1}^n l_i
  $ so there is $1\leq j \leq n$ such that $f(l_k)\cap l_j\neq \emptyset$ hence $f(C_k) \cap C_j\neq
  \emptyset$. Suppose that there is $1\leq i\neq j \leq n$ such that
  $f(C_k) \cap C_i \neq \emptyset \neq f(C_k) \cap C_j$. Let
  $u\in f(C_k)\cap C_i, v\in f(C_k)\cap C_j$, then the subset
  $H:=C_i \cup [u,v] \cup C_j$ is connected included in $M$ since $[u,v]\subset M$.
   Since  $f(C_k)\cap
  F_a=\emptyset$ for any $1\leq k \leq n$ then $H \cap F_a=\emptyset$. By maximality of $C_i$
  and $C_j$ we have $H \subset C_i$ and $H \subset C_j$ hence $H=C_i
  =C_j$, absurd.\\

  $(4)$ Let $1\leq k \leq n$, for any $i\neq \sigma(k), f(C_k) \cap C_i=\emptyset$
then $f(l_k) \cap l_i=\emptyset$, since $f(l_k)\subset
L=\cup_{r=1}^{n} l_r$ hence $f(l_k)\subset l_{\sigma(k)}$. Suppose
 contrarily that $f(l_k)\varsubsetneq l_{\sigma(k)}$, since $f(L)=L$ then
 there is $i\neq k$ such that $f(l_i)\cap l_{\sigma(k)} \neq \emptyset $ hence
   so $\sigma(i)=\sigma(k)$. We obtain $f(L)=\cup_{r=1}^n f(l_r)\subset
 \cup_{r=1}^n l_{\sigma(r)}=\cup_{r=1;r\neq i}^n l_{\sigma(r)} \varsubsetneq L$, absurd.
 Then $f(l_k)= l_{\sigma(k)}$.\\

$(5)$ Suppose that for some $1\leq i\leq n$, $l_i\varsubsetneq
\overline{l_i}$. Then there is $j\in \{1,2,\dots,n\}; j\neq i$ such
that $ \overline{l_i} \cap l_j \neq \emptyset$ hence
    $\overline{C_i} \cap C_j \neq \emptyset$. Since $C_i$ and $C_j$ are disjoint then by Lemma \ref{most}
  there is $z\in M$ such that $\overline{C_i} \cap C_j = \{z\}$. Since $C_i \subset (C_i \cup \{z\}) \subset \overline{C_i}$ then $C_i \cup \{z \}$ is connected, hence $C_i \cup \{ z\} \cup C_j = C_i \cup C_j$ is connected disjoint with $F_a$, this contradict the maximality of $C_i$ and $C_j$.  Then $l_i= \overline{l_i}$ i. e. any $l_i$ is closed in $X$ hence in $L$. In the other hand, since $l_1,l_2,\dots,l_n$ are pairwise disjoint then
$l_i=L \backslash \cup_{j=1;j\neq i}^n l_j$ is open relatively to
$L$. We conclude that any $l_i$ is a clopen set relatively to $L$.\\

$(6)$ Let prove that $\sigma$ is $n$-cycle.  Suppose contrarily that
for some $1\leq s \leq n$, $\sigma^{p}(s)=s$ with $0<p<n$. The
subset $F:=\cup_{k=0}^{p-1}l_{\sigma^k (s)}$ is proper and clopen
relatively to $L$. We have $f(F)=\cup_{k=0}^{p-1}f(l_{\sigma^k
(s)})=\cup_{k=0}^{p-1}l_{\sigma^{k+1}(s)}=F$. Let $G=L \backslash F$
a non empty closed subset in $L$, $\overline{f(L\backslash G)} \cap
G= \overline{f(F)}\cap G=F\cap G=\emptyset$. This contradict Lemma
\ref{WI}. Hence $\forall s \in \{1,2,\dots,n\}$,
$O_{\sigma}(s)=\{1,2,\dots,n\}$ then $\sigma$ is an $n$-cycle.

\end{proof}

\begin{lem} \label{S}
 Denote by $s_1:=[l_1]\subset C_1$. The subset
$S:=\cup_{k=0}^{+\infty}f^{kn}(s_1)$ satisfy conditions $(i)-(iv)$
of the Proposition \ref{prop}.
\end{lem}

\begin{proof}

 First, $S$ is connected. Indeed, since $s_1$ is a subdendrite of $X$ then it is so for
$f^{kn}(s_1), \forall k\geq 0$. We have $f^{n}(s_1)\supset f^n
(l_1)=l_{\sigma^n(1)}=l_1$ hence $f^n (s_1) \supset s_1$, so
$(f^{kn}(s_1))_{k\geq 0}$ is an increasing sequence of  connected
subsets then $S$ is connected.

$(i)$ $S, f(S),\dots,f^{n-1}(S)$ are pairwise disjoint. Indeed,
since $f^{-1}(F_a)=F_a$ and $s_1 \cap F_a=\emptyset$ then
$f^{i}(s_1)\cap F_a=\emptyset , \forall i\geq 0$ hence $f^{k}(S)
\cap F_a=\emptyset ; \forall 0\leq k \leq n-1$. Suppose that
$f^{i}(S)\cap f^{j}(S)\neq \emptyset$ for some $0\leq i \neq j \leq
n-1$ then $f^{i}(S)\cup f^{j}(S)$ is connected.  Let $u\in
l_{\sigma^i(1)}, v\in l_{\sigma^j(1)}$, since $f^k (S)\supset
l_{\sigma^k(1)}$ for any $k\geq0$ then $[u,v]\subset (f^{i}(S) \cup
f^{j}(S)) \cap M$ . The subset $K:= C_{\sigma^i (1)}\cup [u,v]\cup
C_{\sigma^j (1)}$ is connected in $M$ disjoint with $F_a$, by
maximality of $C_{\sigma^i (1)}$ and $C_{\sigma^j (1)}$ we obtain
$K\subset C_{\sigma^i (1)}$ and $K \subset C_{\sigma^j (1)}$ hence
$C_{\sigma^i (1)}=C_{\sigma^j (1)}$, absurd.\\
$(ii)$  $f^{n}(S)=\cup_{k=0}^{+\infty}f^{(k+1)n}(s_1)=
\cup_{k=0}^{+\infty}f^{nk}(s_1)=S$.\\
$(iii)$ Since $l_1 \subset S$ then $L=\cup_{k=0}^{n-1}l_{\sigma^k(1)}\subset \cup_{k=0}^{n-1}f^{k}(S)$.\\
$(iv)$ $ L\cap f^{k}(S)=l_{\sigma^k(1)}$ for any $0\leq k \leq n-1$.
Indeed, we have $f^{k}(S) \supset f^{k}(l_k)=l_{\sigma^k (1)};
\forall k\ge0$ and since $S, f(S),\dots,f^{n-1}(S)$ are pairwise
disjoint then $L\cap f^{k}(S)=l_{\sigma^k(1)}$, so $f(L\cap
f^{k}(S))=f(l_{\sigma^k (1)})=l_{\sigma^{k+1} (1)}=L \cap
f^{k+1}(S)$.
\end{proof}
 This finish the proof of Case 1.

 \bigskip

\textbf{Case 2:  $M\cap Fix(f)=\emptyset$}.\\
In the following Lemma we will use the notations from \cite{MS}.

\begin{lem}\label{fix}
  If $M$ does not contains a fixed point then there is a fixed point
  $z\in X\backslash M$ such that $M\subset \psi^{-1}_f(z)$.
\end{lem}

\begin{proof} Let $y\in M$ such that $\psi_f(y)=z \in Fix(f)$.
Suppose contrarily that $M\nsubseteq \psi_f ^{-1}(z)$ then $M \cap
(X \backslash \psi_f ^{-1}(z))\neq \emptyset$ since $M \cap \psi_f
^{-1}(z)\neq \emptyset$ and $M$ is connected then $M\cap
\partial(\psi_f ^{-1}(z))\neq \emptyset$ but $\partial(\psi_f
^{-1}(z))\subset Fix(f)$ hence $M\cap Fix(f)\neq \emptyset$, absurd.
Then $M\subset \psi_f ^{-1}(z)$.  It follow that $\forall y\in M,
\psi_f(y)=z$. This finish the proof of the Lemma.
\end{proof}

 Now, let $z_0:=r_M(z)$ where $r_M$ is the retraction map
 (See \cite{Nad}) and denote by
 $F_{z_0}:=\cup_{n=0}^{+\infty}f^{-n}(z_0)\cap M$ and
 $Y_0:=[\overline{F_{z_0}}]$. We remark that $F_{z_0}\cap L=\emptyset$
  (so $ord(z_0,M)\geq 2$), $F_{z_0}$ is non empty
  and any arc
 joining a point in $M$ and $f(z_0)$ contains $z_0$.

  We have an analogous results compared
 to Case $1$.

 \begin{lem} Let Assumption \ref{L1} be satisfied. Then we have the following properties:
 \begin{enumerate}
 \item $L\subset M\backslash Y_0$,
 \item $M\backslash Y_0$  has finitely many connected components.
 The connected components of
$M\backslash F_{z_0}$ intersecting $L$ are  denoted  by $C_1,
C_2,\dots, C_n$ where $n>1$,we denote by  $l_k= L \cap C_k$,
\item for any $1\leq k \leq n$, there is a unique $j:=\sigma(k)\in
\{1,2,\dots,n\}$ such that $f(C_k) \cap C_j \neq \emptyset$,
\item for any $1\leq k \leq n$, $f(l_k)=l_{\sigma(k)}$,
\item for any $1\leq k \leq n$, $l_k$ is a clopen set relatively to
$L$,
\item $\sigma$ is an $n$-cycle,
\item Let $S=\cup_{k=0}^{+\infty}f^{kn}([l_1])$. Then
 \begin{enumerate}
\item  $S, f(S),\dots, f^{n-1}(S)$ are a pairwise disjoint connected
 subsets,
 \item $f^{n}(S)=S$,
 \item $L\subset \cup_{i=0}^{n-1}f^{i}(S)$
 \end{enumerate}

 \end{enumerate}
 \end{lem}

 \begin{proof}

 $(1)$
\textbf{Claim 1.} There is $n_0 \geq0$ such that $\forall n\geq n_0$
we have $z_0 \in (f(z_0),f^n(x))$.\\ Since $z_0 \notin L$ and $L$ is
compact then $\delta:=d(z_0,L)>0$. Remark that for any $y\in L, z_0
\in (y,f(z_0))$. Then $\forall y\in L, \exists 0<\mu=\mu(y)<\delta$
such that $\forall t\in B(y,\mu), z_0\in (t,f(z_0))$. The subset
$V:=\cup_{y\in L}B(y,\mu(y))$ is an open neighborhood of $L$
disjoint with $z_0$. There is $n_0\geq 0$ such that $\forall n\geq
n_0, f^{n}(x) \in V$. Then $\forall n\geq n_0$, there is $y\in L,
f^n(x) \in B(y,\mu(y))$ hence $z_0 \in (f^n(x),f(z_0)$.

\medskip

 \textbf{Claim 2.} For any $t\in F_{z_0}$, $[z_0,t]\cap L=\emptyset$. \\
 Suppose contrarily that there is $t\in f^{-k}(z_0)
\cap M ; k>0$ and $y\in L$ such that $y\in [z_0,t]$. Denote by
$I=[z_0,y]$, $J=[y,t]$. We have $f^{k}(J)\supset [z_0,f^{k}(y)]$
then $f^{k+1}(J)\supset [f(z_0),f^{k+1}(y)]\ni z_0$ hence
$f^{k+1}(J)\supset [z_0, f^{k+1}(y)]$ so by induction, for any
$n\geq k, f^{n}(J)\supset [z_0,f^n(y)]$. Similarly, we have
 $f(I)\supset [f(z_0),f(y)]\ni z_0$ then $f(I)\supset [z_0,f(y)]$.
 Hence we prove
 by induction that for any $n\geq0$, $f^{n}(I)\supset [z_0,f^{n}(y)]$.
As the proof of Lemma \ref{L3}, Claim 1, we show easily that $I, J$
form an arc
 horseshoe hence $h(f)>0$, a contradiction. \\

 \textbf{Claim 3.} $\overline{F_{z_0}}\cap L=\emptyset$. \\
 Suppose contrarily that there is
 $y\in L\cap\overline{F_{z_0}} $. Let $(z_{i})_{i>0}$ be a
 sequence in $F_{z_0}$ converging to $y$ and
 $f^{n_i}(z_i)=z_0; \forall i>0$.
 Let prove that there is a neighborhood
 $V$ of $y$ in $X$ and $k\geq0$  such that $f^k(V)\subset I$ where
 $I$ is a free open arc in $X$ such that $
 [z_0,s]\cap L\neq \emptyset$ for some $s\in F_{z_0}$. Indeed, since $\omega_f(y)$ is uncountable
 then there is a free open arc $I$ in $X$ such that $I\subset M$ and
 $\omega_f(y)\cap I$
 is uncountable. Let $y_1, y_2 $ two distincts points in the intersection
 such that $y_1 \in (z_0,y_2)$ and $I_1, I_2$ two open disjoint arc in $I$ containing $y_1, y_2$
 respectively. Let  $k,p>0$ such that $f^k(y)\in I_1, f^p(y)\in I_2$.
 By continuity of $f^p$ there is an open set, $V$, in $X$ containing $y$ such that
 $f^p(V)\subset I_2$. If the sequence $(n_i)_{i>0}$ is bounded then by
 taking a subsequence we may assume that $n_i=:m$ for any $i>0$. So we have
 $f^m(z_i)=z_0, \forall i>0$, since $z_i \to_{i\to +\infty} y$
 then $f^m(y)=z_0\in L$, absurd. If the sequence $(n_i)_{i>0}$ is unbounded,
  we may assume that it is non decreasing and $n_i>p, \forall i>0$.
  Let $i>0$ such that $z_i \in V$ then $f^p(z_i) \in f^p(V)\subset I_2
\subset M$. So
  we have $f^k(y)\in (z_0,f^p(z_i))$ and $f^{n_i -p}(f^{p}(z_i))=z_0$ then
  $f^p(z_i)\in F_{z_0}$, this contradict Claim 2. Hence we finish
the proof of Claim 3.\\
  Now, we have $L\subset M\backslash [z_0,t], \forall t\in \overline{F_{z_0}}$
then $L\subset M\backslash \cup_{t\in
\overline{F_{z_0}}}[z_0,t]=M\backslash Y_0$.

$(2)$ $Y_0$ is a subdendrite of $M$ satisfying  $E(Y_0)\cap
E(M)^{\prime}\subset \overline{F_{z_0}} \cap L=\emptyset $ then
$M\backslash Y_0$ has finitely many
connected components.\\
The proof of $(3)-(7)$ are similar as in Lemmas \ref{C} and \ref{S}.

\end{proof}

%\hfill $\square$

 \subsection{\textbf{Proof of Theorem B}}

Let $S$ be a connected subset of $X$ satisfying conditions
$(i)-(iii)$ of proposition \ref{prop}. Denote by $\alpha_1:=n_1$ and
$D_1=\overline{S}$. Then we have $L\subset \cup_{i=0}^{\alpha_1
-1}f^{i}(D_1)$.
 Set $f_1:=f^{\alpha_1}_{|D_1}:~D_1 \to D_1 $, it is a
  dendrite map such
 that $E(D_1)$ is  countable closed  and $E(D_1)^{\prime}$  finite. Let $t_1=f^{p_1}(x)\in D_1$. Since $\omega_{f_1}(t_1)$ is uncountable then by Proposition \ref{prop} there is a pairwise disjoint connected subsets $S_2, f_{1}(S_2),\dots, f_{1}^{n_2
 -1}(S_2); n_2\geq2$, $f_{1}^{n_2}(S_2)=S_2$ in $D_1$ such that $\omega_{f_1}(t_1)
 \subset \cup_{i=0}^{n_2 -1}f_{1}^{i}(S_2)$. So for any $0\leq j \leq n_1 -1,
 \omega_{f^{n_1}}(f^{j}(t_1))\subset \cup_{i=0}^{n_2 -1}f^{i n_1
 +j}(S_2)\subset \cup_{i=0}^{n_2 -1}f^{i n_1
 +j}(D_2)$ where $D_2=\overline{S_2}$. Hence $L=\omega_f (t_1)\subset
 \cup_{j=0}^{n_1 -1} \cup_{i=0}^{n_2 -1}f^{i n_1
 +j}(D_2) =\cup_{i=0}^{\alpha_2 -1}f^{i}(D_2)$. The subdendrite $D_2$ has period $\alpha_2:=n_1
 n_2$ and $\cup_{k=0}^{n_2 -1} f^{i \alpha_1}(D_2)\subset D_1$.
 Since card$(f^{i}(D_1)\cap f^{j}(D_1))\leq 1$ for any $0\leq i \neq j \leq \alpha_1
 -1$ then card$(f^{i}(D_2)\cap f^{j}(D_2))\leq 1$ for any $0\leq i \neq j \leq \alpha_2
 -1$. By induction we build a sequence of subdendrites
 $(D_k)_{k\geq1}$ satisfying the conditions of Theorem B.\\

%By Proposition \ref{thm1},
% then it intersect a free open arc
%$I$ in some $f^{i}(D_1)$ hence $\emptyset \neq O_f(x)\cap I \subset
%O_f(x) \cap f^{i}(D) $.

\section{\textbf{On Li-Yorke pair implies chaos.}}

To prove Theorem C we need the following Lemma. \\
\begin{lem}
We use the notations from Theorem B. There is $k>0$ and $0\leq i <
\alpha_k$ such that $f^{i}(D_k)$ is a free arc.
\end{lem}
\begin{proof}
 \textbf{Case 1.} If $X$ is a tree. The number $N:=\sum_{b\in B(X)} ord(b,X)$ is finite, let $k>0$ such that $\alpha_k >N$. Then there is
 $0\leq i <\alpha_k$ such that $f^{i}(D_k) \cap B(X)=\emptyset$ then
 $f^{i}(D_k)$ is a free arc in $X$.

 %We will prove the result by
% induction on $n=$card$E(X)\geq 2$. The
% result is true when $X$ is an arc. Suppose that it remains true for any tree $X$ with
% card$E(X)=i, \forall i\in \{1,2,\dots,n-1\}; n>2$ and suppose that
% card$E(X)=n$. Since
% card$(f^{k}(D_1) \cap f^{l}(D_1))\leq 1; \forall 0\leq k \neq l <\alpha_1$
% then there is $0\leq j < \alpha_1$ such that $f^{j}(D_1)$ is a tree
% with card$(E(f^{j}(D_1)))\leq n-1$. By hypothesis there is $k>1$
% and $0\leq l <\alpha_k$ such that $f^{l}(D_k) \subset f^{j}(D_1)$
% is an arc.
\medskip

 \textbf{Case 2.} If $X$ is not a tree. Let $d:=card(E(X)^{\prime})\geq
 1$. For any $k\geq1$ we have $n_k \geq 2$ so $\alpha_k \geq 2^k$.
 Let $k>0$ such that $d<2^k$. Since $card(f^{i}(D_k)\cap f^{j}(D_k))\leq
 1$ then there is $0\leq i <\alpha_k$ such that $f^{i}(D_k)\cap
 E(X)^{\prime}=\emptyset$. By Lemma \ref{YD}, $f^{i}(D_k)$ is a tree. By Case 1, there is
 $p\geq k$ and $0\leq j <\alpha_p$ such that $f^{j}(D_p)$ is an arc
 included in $f^{i}(D_k)$. Since $f^{j}(D_p)\cap E(X)^{\prime}=\emptyset$
 then the set $f^{i}(D_p)\cap B(X)$ is finite. So there is $q>p$
 and $0\leq s<\alpha_q$ such that $f^{s}(D_q)$ is a free arc in
 $X$.
 \end{proof}
\medskip

Proof of Theorem \textbf{C}. Let $(x,y)$ be a Li Yorke pair for $f$
such that $L:=\omega_f (x)$ is uncountable. Denote by
$I=f^{i}(D_k)=[u,v]$ a free arc in $X$. There is $n\geq 0$ such that
$x_n \in I$. Denote by $d=\alpha_k$, then $g=f^{d}_{|I}: I \to I$ is
an interval map. Since $(x,y)$ is a Li Yorke pair for $f$ then
$(x_n, y_n)$ is proximal for $f^{d}$. So $\omega_{f^d}(x_n) \cap
\omega_{f^d}(y_n)\neq \emptyset$. We
distinguish two cases: \\
\medskip

\textbf{Case 1:}  $O_{f^{d}}(y_n)\cap I \neq \emptyset $.  Let
$k\geq 0$ such that $f^{dk}(y_n) \in I$. So
$(f^{dk}(x_n),f^{dk}(y_n))\in I^2$ is a Li Yorke pair for $g:~I \to
I$, by \cite{KS} $g$ is chaotic hence $f$ is chaotic.\\
\medskip

\textbf{Case 2.}  $O_{f^{d}}(y_n)\cap I = \emptyset $. Since
$\omega_{f^d}(x_n) \cap \omega_{f^d}(y_n)\neq \emptyset$ and
$\omega_{f^d}(x_n)\subset I$ then $\omega_{f^d}(y_n) \cap I \subset
\{u,v\}$ so $\omega_{f^d}(x_n) \cap \omega_{f^d}(y_n)$ is finite and
$f^d$-invariant so $\omega_{f^d}(x_n)$ contains a periodic point,
say $u$. Then  $(x_n,u)\in I^2$ is a Li Yorke pair for $g$, by
\cite{KS} $g$ is chaotic, hence $f$ is chaotic.

\smallskip

\section{\textbf{Example}}

\textbf{Example 4.} (Due to I. Naghmouchi)
 We will construct  a dendrite $X$ with $E(X)$ is  countable closed set such that $(E(X))^{(2)}$
is reduced to one point and a map $f$ on $X$  having a Li-Yorke pair
but not Li-Yorke chaotic.

\begin{proof}
 \textbf{ Construction of the dendrite $X$.} \\ Denote by
$\displaystyle D_0= ([-1,1]\times \{0\}) \cup (\bigcup_{n\geq
1}\{1-\frac{1}{n}\}\times[0,\frac{1}{n}]).$
  Let $I=[(0,0),(\frac{1}{2},-1)]$ and $X=I\cup (\cup_{n\geq 1}X_n)$
  where for any $n\geq 1$,
  \begin{itemize}
  \item $X_n$ is a homeomorphic copy of $D_0$, by a homeomorphism $\varphi_n$,
  \item  $a_k^n:=\varphi_n((1-\frac{1}{k},\frac{1}{k}))$, $\forall k\geq 1$ and $a^n=\varphi_n((1,0))$, $a^0=(0,0)$,
  \item $X_n \cap I=\{b_n\}$ where $\varphi_n((-1,0))=b_n$,
  \item the sequence $(X_n)_{n\geq1}$ are pairwise disjoint and $\lim_{n\to +\infty}$diam$(X_n)=0$,
  \item $ c_k^n=\varphi_n((1-\frac{1}{k}, 0))$, $\forall n, k \geq 1$.
  \end{itemize}

  We can see that the set $X$ is a dendrite such that $E(X)=\{a_k^n; n,k\geq1 \}\cup
  \{ a^n; n\geq 1\}\cup\{a^0 \}$,  $E(X)^{\prime}=\{a^n; n\geq 1\}\cup \{a^0 \}$ and
   $E(X)^{(2)}=\{a^0 \}.$

  %For any $n\geq 1$, let
%$(a_{k}^{n})_{k\geq 1}$ be a non decreasing sequence in the interval
%$[0,1]$ which converges to $\frac{1}{n}$ and $\frac{1}{n+1}<a_{1}^n
%$.

\medskip

  \textbf{ Construction of the map $f$.} (See Figure $3$) \\
 The map $f$  is defined as  follows:
 \begin{itemize}
 \item $f(a^1)=f(a^0)=a^0$ and $f(b_1)=b_1$,
 \item $f([b_2,c_1^1])=\{b_1\}$,
 \item $\forall n\geq1, f([b_n,b_{n+1}])=\{b_n\},$
 \item $\forall n, k\geq 1, f$ sent linearly $[b_{n+1},a^{n+1}]$ to
 $[b_{n},a^n]$ such that $ f(c_k^{n+1})=c_{k+1}^{n}$ and $f$ sent linearly $[c_k^{n+1},a_k^{n+1}]$ to $[c_{k+1}^{n},a_{k+1}^{n}]$,
 \item $\forall k\geq 1, f$ sent linearly $[c_1^1,a^1]$ to $[b_1,
 a^0]$ such that  $f(c_k^1)=b_k$ and $f$ sent linearly $[c_k^1, a_k^1] $ to $[b_k, a_{1}^{k+1}].$
 \end{itemize}

We see that $f$ is a dendrite map   having a Li-Yorke pair (take for
example $(a^0,a_1^1)$ ) but $f$ is not not Li-Yorke chaotic. Indeed,
if $(x,y)\in X \times X$ is a Li yorke pair then $(x,y)\in
E(X)^{2}$.

%if, and only if, $x,y\in E(X); x\neq y \text{ and } (x,y) \notin
%E(X)^{\prime}$. Furthermore, $\forall z\in X \backslash E(X),
%\lim_{n\to +\infty} f^{n}(z)=b_1.$
\end{proof}

\begin{figure}[H]
\begin{center}
\includegraphics[width=14cm,height=8cm]{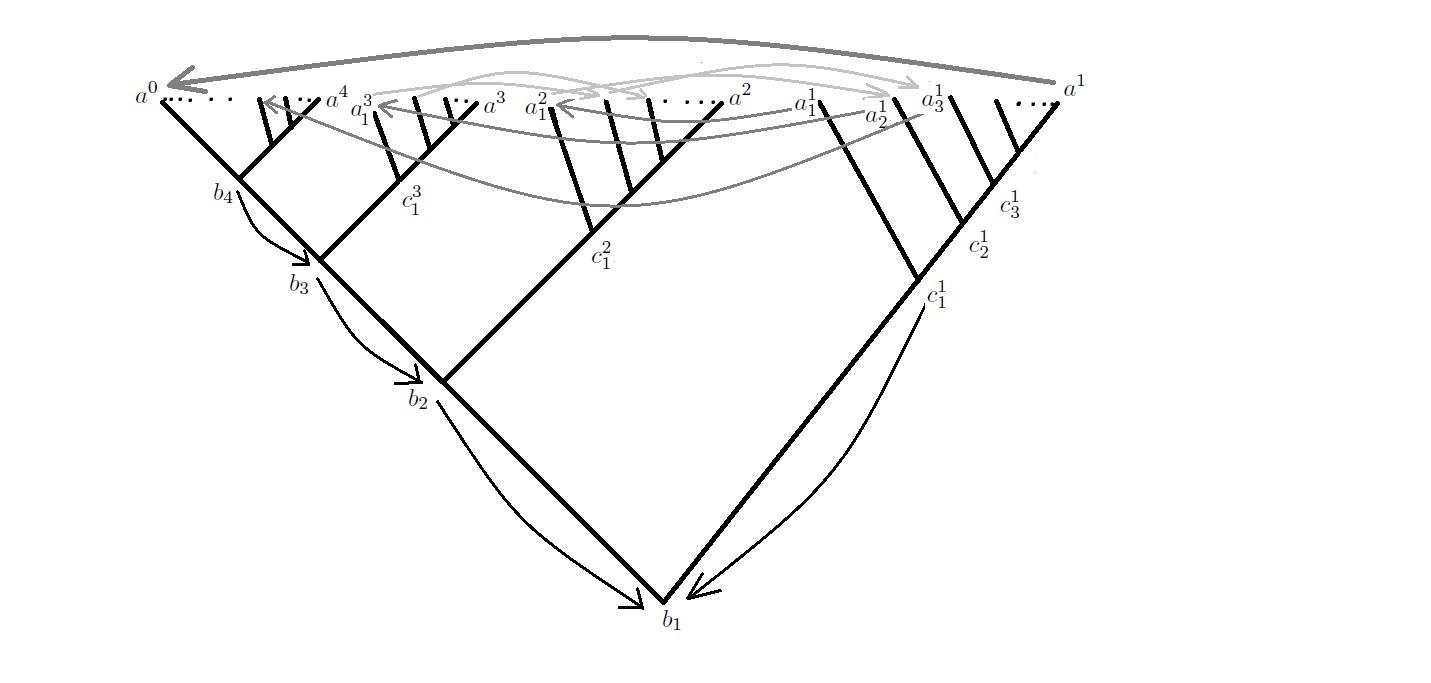}
\caption{Dendrite with $E(X)^{\prime }$ is infinite}
\end{center}
\end{figure}

\begin{lem} \label{D^2}
Every dendrite $X$ with closed set of endpoints and $E(X)^{(2)}$ is
 non  empty  contains a subdendrite $Y$ with closed set of endpoints and $E(Y)^{(2)}$ is reduced to one point.
\end{lem}

\begin{proof}
\medskip

 \textbf{Claim.}
 \emph{If $F$ is a closed set such that $F^{(2)}$ is non empty then it contain a closed subset $A$ such that $A^{(2)}$ is reduced to one point.}\\
  Let $e\in F^{(2)}$, there is a sequence of pairwise different elements  $(e_n)_{n\geq0}$ in $F^{\prime }$ converging to
 $e$. Let $(\varepsilon_n)_{n\geq 0}$ a sequence of non negative reels such that
 the sequence of balls $(B(e_n,\varepsilon))_{n\geq 0}$ are pairwise disjoint.
For any $n\geq0$, since $e_n\in F^{\prime}$ then there is a sequence
$(e_n^k)_{k\geq0}$ in $B(e_n,\varepsilon_n)\cap F$ converging to
$e_n$. Hence the set $A=\{e_n^k; n,k\geq 0\} \cup \{e_n; n\geq 0 \}
\cup \{e\}$ satisfy the condition of the Claim.
 Now, let $A$ be a closed subset in $E(X)$ such that $A^{(2)}$ is one point. By Lemma \ref{hull}, $Y:=[A]$ is a subdendrite of $X$ with $E(Y)=A$.
\end{proof}

 %\textbf{Claim 2.}
% \emph{If $A$ is a closed set such that $A \subset E(X)$
% then the convex hull of $A$, $[A]$, is a dendrite with $E([A])=A$.}\\
% Fix $a\in A$, write $[A]=\cup_{b\in A }[a,b]$. By Lemma \ref{hull}, we have $E([A])\subset A$. Conversely, since $E([A])=[A]\cap E(X)$ and $A\subset [A] \cap E(X)$
% then $A\subset E([A])$ so $A=E([A])$.

 %\quad \quad \quad $\square$
 By Lemma \ref{D^2}, there is a non chaotic dendrite map $g:~Y \to Y$ having a
 Li-Yorke pair $(x,y)\in Y^2$. Let $r_Y:~X\to Y$ be the retraction map
 and $f=g \circ r_Y:~X \to Y \subset X$. Then $(x,y)$ is also a Li Yorke pair for $f$
 but $f$ is not Li-Yorke chaotic. \hfill $\square $

\bigskip

%\textbf{Acknowledgments:}  I would like to thanks Professor Habib Marzougui for helpful discussions on the subject of the paper.

 \end{document}